\documentclass[a4paper, 11pt, english]{article}
\usepackage[utf8]{inputenc}
\usepackage[T1]{fontenc}
\usepackage{graphicx}
\usepackage{stmaryrd}
\usepackage[a4paper]{geometry}
\usepackage{amsmath,amsfonts,amssymb,amsthm,epsfig,epstopdf,url,array}
\usepackage{rotating}
\usepackage[colorlinks=true,citecolor=red,linkcolor=blue,pdfpagetransition=Blinds]{hyperref}
\usepackage{cleveref}
\usepackage{nameref}
\usepackage{enumitem}
\usepackage{comment}
\Crefname{paragraph}{Section}{Sections}
\usepackage{fancyhdr}
\usepackage{xcolor}

\usepackage{mathrsfs}

\usepackage{color}

\newcommand\rouge[1]{{\bf\color{red} #1}}

\usepackage{fullpage}
\crefname{theo}{theorem}{theorems}

\usepackage[sort,nocompress]{cite}

\providecommand{\keywords}[1]{\noindent {\textit{Keywords:}} #1}

\theoremstyle{plain} 
\newtheorem{proposition}{Proposition}[section] 
\newtheorem{theo}[proposition]{Theorem}

\newtheorem{lemma}[proposition]{Lemma}

\newtheorem{rmk}[proposition]{Remark}
\theoremstyle{definition}
\newtheorem{defi}[proposition]{Definition}

\DeclareMathOperator*{\esssup}{ess\,sup}

\def\dx{\textnormal{d}x}
\def\dt{\textnormal{d}t}
\def\d{\textnormal{d}}
\def\E{{\mathbb{E}}}
\def\dom{\mathcal{O}}
\def\F{\mathcal{F}}

\def\P{\mathbb{P}}
\def\fil{\mathbb{F}}

\def\bigmid{\;{\Big |}\;}

\newcommand{\norme}[1]{\left\lVert#1\right\rVert}

\newcommand{\dive}[1]{\mathrm{div}}
\newcommand{\ov}[1]{\overline{#1}}
\def\dx{\,\textnormal{d}x}
\def\dt{\textnormal{d}t}
\def\d{\textnormal{d}}
\def\E{{\mathbb{E}}}
\def\R{{\mathbb{R}}}
\def\dom{\mathcal{O}}
\def\F{\mathcal{F}}

\makeatletter
\let\original@addcontentsline\addcontentsline
\newcommand{\dummy@addcontentsline}[3]{}
\newcommand{\DeactivateToc}{\let\addcontentsline\dummy@addcontentsline}
\newcommand{\ActivateToc}{\let\addcontentsline\original@addcontentsline}
\makeatother

\begin{document}

\title{$L^{p}$-estimates, local well-posedness and controllability for linear and semilinear backward SPDEs}

\author{V\'ictor Hern\'andez-Santamar\'ia\thanks{The work of V. Hern\'andez-Santamar\'ia is supported by the program ``Estancias Posdoctorales por México para la Formación y Consolidación de las y los Investigadores por México'' of SECIHTI (Mexico). He also received support from Project CBF2023-2024-116 of SECIHTI and by UNAM-DGAPA-PAPIIT grants IN117525, IA100324 and IN102925 (Mexico).} \and  K\'evin Le Balc'h\thanks{The work of K. Le Balc'h is partially supported by the Project TRECOS ANR-20-CE40-0009 funded by the ANR (2021–2024).} \and Liliana Peralta\thanks{L. Peralta has received support from UNAM-DGAPA-PAPIIT Grants IA100324 and IA103826.}}

\maketitle

\begin{abstract}
In this paper, we study linear backward parabolic SPDEs {in bounded domains} and present new a priori estimates for their weak solutions. Inspired by the seminal work of Y. Hu, J. Ma and J. Yong from 2002 on strong solutions, we establish $L^p$-estimates requiring minimal assumptions on the regularity of the coefficients, the terminal data, and the external force. {Our approach relies on direct, constructive, and quantitative arguments, adapted from known methods in the theory of SPDEs to this setting.  In particular, we develop a new Itô’s formula for the $L^p$-norm of the backward solution, tailored to this setting and extending the classical result in the $L^2$-framework}. This formula is then used  to improve further the regularity of the first component of the solution up to $L^\infty$. {We also present two applications: a local existence result for a semilinear equation without imposing any growth condition on the nonlinear term, and a novel local controllability result for semilinear backward SPDEs that partially resolves an open problem in the field.
}
\end{abstract}


\small
\tableofcontents
\normalsize

\section{Introduction}

\subsection{Notation and main results}

Backward stochastic partial differential equations (BSPDEs, for short) arise in many applications of probability theory and stochastic processes. In the literature, BSPDEs can be found in various contexts, including control under incomplete information (see \cite{Ben83,Tan98}), as adjoint equations in optimal control (see e.g. \cite{NN90,HP91,Zho93}), in controllability problems (see e.g. \cite{BRT03,TZ09,HSLBP23}), in mathematical finance for the formulation of the stochastic Feynman-Kac formula and the Black-Scholes equation (see \cite{MY97,MY99}), among other applications.

Let $(\Omega,\F,\P)$ be a complete probability space on which is defined a one-dimensional standard Brownian motion $W=\{W(t):t\geq 0\}$ and let $\{\F_t\}_{t\geq 0}$ be the natural filtration generated by $W$, augmented by all the $\P$-null sets in $\F$. For any given (deterministic) time $T>0$, let us consider the following linear BSPDE
\begin{equation}
\label{eq:intro}
\begin{cases}
\d{y}=-\left(\Delta y + \alpha y + \beta Y + F\right)\d{t}+Y\d{W}(t) &\text{in } (0,T)\times\dom, \\
y=0 &\text{on } (0,T)\times\partial\dom, \\
y(T)=y_T &\text{in }\dom,
\end{cases}
\end{equation}
where $\mathcal O\subset\mathbb R^d$ ($d\geq 1$) is an open and bounded set with a $C^2$ boundary $\partial \dom$. In \eqref{eq:intro}, the coefficients $\alpha$ and $\beta$, the forcing term $F$ and the terminal datum $y_T$ are suitable random fields verifying appropriate measurability and regularity conditions. 

A solution to \eqref{eq:intro} consists of a pair of random fields $(y,Y)$ that are $\{\F_t\}_{t\geq 0}$-adapted and satisfy \eqref{eq:intro} in a specified sense. At this point, it is important to note a fundamental difference from standard SPDEs: the noise term in a BSPDE like \eqref{eq:intro} arises from the intrinsic randomness in the coefficients and terminal data---making it endogenous---and comes from martingale representation theorems. 

A classical question for BSPDEs consists in understanding the notion of the solution (for instance strong, mild, or weak in the PDE sense) and determining the best regularity available for the given data of the problem. This issue has been studied in the literature in different contexts and frameworks. In \cite{HP91,LvN19}, using semigroup theory, the authors establish the existence and uniqueness of mild solutions for abstract backward evolution equations in the $L^2$ and $L^p$ settings, respectively. In \cite{HMY02,Dok12,DT12}, the problem of existence and regularity of strong solutions (i.e., solutions that can be evaluated point-wisely in the spatial variable) is studied in bounded and/or unbounded domains. In the context of weak solutions (i.e., the PDE satisfies a scalar product equation), the works \cite{DQT12,DTZ13} tackle the existence and regularity of solutions posed in the whole space within the $L^p$-setting and in higher order Sobolev spaces $W^{m,p}$ ($m\geq 1$, $p\geq 2$), respectively. Finally, existence and regularity of solutions in H\"older spaces are established in \cite{TW16}.

In this direction, the first goal of this paper is to establish $L^p$- and $L^\infty$-estimates for weak solutions to \eqref{eq:intro} (see \Cref{def:weak} below) without imposing additional regularity and/or differentiability assumptions on the coefficients, the terminal data or the source term. Our aim is to use natural hypotheses corresponding to the desired function space: for instance, obtaining $L^p$ estimates requires $L^p$ terminal datum and source term, and similarly, $L^\infty$ solutions are obtained with $L^\infty$ data.

To make this precise, we begin by introducing the following notations.  In what follows, we denote $\{\F_t\}_{t\geq 0}$ by $\fil$ unless we want to emphasize a precise $\mathcal F_t$. Let $(X,\|\cdot\|_X)$ be a Banach space, for any $p\in[2,+\infty)$ and any $s\in\{0,T\}$, we denote by $L^p_{\mathcal F_s}(\Omega;X)$ the set of all $\F_s$-measurable $X$-valued random variables $\xi:\Omega\to X$ such that $\E(|\xi|_{X}^p)<+\infty$. For any $p,q\in[2,\infty)$ we define the spaces
\begin{equation}\label{LpLq}
\begin{split}
L_{\fil}^p(\Omega;L^q(0,T;X)):=\Big\{\psi & :\Omega\times[0,T]\to X \mid \, \psi(\cdot) \text{ is an }  \fil\text{-adapted process }\\ &\textnormal{ on }[0,T]\text{ and } \E\left[\Big(\int_0^T\|\psi(t)\|^q_{X}\dt\Big)^{\frac pq}\right]<+\infty\Big\},
\end{split}
\end{equation}
\begin{equation}\label{LqLp}
\begin{split}
L_{\fil}^q(0,T;L^p(\Omega;X)):=\Big\{\psi & :\Omega\times[0,T]\to X \mid \, \psi(\cdot) \text{ is an }  \fil\text{-adapted process }\\ &\textnormal{ on }[0,T]\text{ and } \int_0^T\left[\E\left(\|\psi(t)\|^p_{X}\right)^{\frac qp}\right]\dt <+\infty\Big\},
\end{split}
\end{equation}
endowed with their natural norms. In a clear and analogous way, we define the spaces \eqref{LpLq} and \eqref{LqLp} when $p$ and/or $q$ equal to $\infty$. When $p=q$ with $p\in[2,+\infty)$, from Fubini's theorem we have that $L_{\fil}^p(\Omega;L^p(0,T;X))=L_{\fil}^p(0,T;L^p(\Omega;X))$ and the norms coincide and for simplicity we just write $L_{\fil}^p(0,T;X)$. When $p=q=\infty$, we also have that $L_{\fil}^{\infty}(\Omega;L^{\infty}(0,T;X))=L_{\fil}^{\infty}(0,T;L^{\infty}(\Omega;X))$ and the norms coincide. Lastly, by $L_{\fil}^p(\Omega; C([0,T];X))$ we denote the Banach space consisting of all $X$-valued $\fil$-adapted continuous processes $\psi(\cdot)$ such that $\mathbb{E}\left(\norme{\psi}_{C([0,T];X)}^p \right) < +\infty$, also equipped with the canonical norm.

In the remainder of this document, we will make the following instrumental assumption on the coefficients of system \eqref{eq:intro}

\begin{enumerate}[label={\textnormal{\bf (H)}}]
\item \label{Hy} $\alpha,\beta\in L^\infty_{\fil}(0,T;L^\infty(\dom))$.
\end{enumerate}

We recall the following notion of solution for \eqref{eq:intro}.
\begin{defi}\label{def:weak} A pair of random fields $(y,Y)$ is called a weak solution to \eqref{eq:intro} if
\begin{enumerate}
\item[1)] $(y,Y)$ is $L^2(\mathcal O)\times L^2(\mathcal O)$-valued and $\mathcal F_t$-measurable for each $t\in[0,T]$,
\item[2)] $(y,Y)\in \left[L^2_{\fil}(\Omega;C([0,T]);L^2(\dom))\bigcap L^2_{\fil}(0,T;H_0^1(\dom))\right]\times L^2_{\fil}(0,T;L^2(\dom))$, and
\item[3)] for any $t\in[0,T]$ and $\phi\in H_0^1(\mathcal O)$ it holds
\begin{align*}
(y(t),\phi)_{L^2(\mathcal O)}&=(y_T,\phi)_{L^2(\mathcal O)}-\int_{t}^{T}(\nabla y,\nabla \phi)_{L^2(\mathcal O)}\d{s}\\
&\quad +\int_{t}^{T}(\alpha y+\beta Y+F,\phi)_{L^2(\mathcal O)}\d{s}-\int_{t}^{T}(Y,\phi)_{L^2(\mathcal O)}\d{W}(s), \quad \textnormal{a.s.}
\end{align*}
\end{enumerate}
\end{defi}

The following result ensures the existence and uniqueness of such weak solutions.

\begin{theo}\label{thm:l2_sol}
Let $\alpha,\beta$ be given coefficients satisfying \ref{Hy}. If $y_T\in L^2_{\mathcal F_T}(\Omega; L^2(\mathcal O))$ and $F\in L^2_{\fil}(0,T;L^2(\mathcal O))$, then \eqref{eq:intro} admits a unique weak solution in the sense of \Cref{def:weak}. Moreover, the following energy estimate holds
\begin{align}\notag
&\E\left(\sup_{t\in[0,T]}\|y(t)\|^2_{L^2(\dom)}\right)+\E\left(\int_0^T\int_{\dom}|\nabla y|^2\dx\dt\right) +\E\left(\int_0^T\int_{\dom}|Y|^2\dx\dt\right) \\ \label{thm:l2_reg}
&\quad \leq C\E\left(\int_{\dom}|y_T|^2\dx+\int_0^T\int_{\dom}|F|^2\dx \dt\right),
\end{align}
for a positive constant $C>0$ independent of $y_T$ and $F$.
\end{theo}

This theorem is well-known in the literature and the proof can be derived in two different ways: by a duality analysis similar to \cite{zhou92} or by a more direct and constructive approach relying on Galerkin method (see, e.g., \cite[Proposition 2.1]{Gao18} or \cite[Section 3]{SY09}).

Our first main result says that the integrability of the weak solution $(y,Y)$ can be improved by taking more regular data. More precisely, we have the following. 

\begin{theo}\label{thm:ito}
Let $(y,Y)$ be a weak solution to \eqref{eq:intro}. Assume that $y_T\in L^p_{\mathcal F_T}(\Omega;L^p(\dom))$ and $F\in L^{p}_{\fil}(0,T;L^p(\mathcal O))$ for some $p\in[2,+\infty)$. Then, $$(y,Y)\in L^p_{\fil}(\Omega;C([0,T];L^p(\dom)))\times L^p_{\fil}(\Omega;L^2(0,T;L^2(\dom)))$$ and satisfies
\begin{align}\notag
&\E\left(\sup_{t\in[0,T]}\|y(t)\|^p_{L^p(\dom)}\right)+ \E\left[\left(\int_0^{T}\|Y(t)\|_{L^2(\dom)}^2\dt\right)^{p/2}\right] \\ \label{est_lp}
&\qquad \leq Ce^{C T}\E\left(\int_{\dom}|y_T|^p\dx+\int_0^T\int_{\dom}|F|^p\dx \dt\right),
\end{align}
for a constant $C>0$ only depending on $\dom$, $\alpha$, $\beta$ and $p$.
\end{theo}


Our second result says that we can go up to $L^\infty$ for the first component $y$ of the solution $(y,Y)$ to \eqref{eq:intro}. In more detail, we have the following.

\begin{theo}\label{thm:infinity}
Let $(y,Y)$ be a weak solution to \eqref{eq:intro}. Assume that $y_T\in L^\infty_{\mathcal F_T}(\Omega;L^\infty(\dom))$ and $F\in L^{\infty}_{\fil}(0,T;L^\infty(\mathcal O))$. Then the process $y$ taken from the solution $(y,Y)$ to \eqref{eq:intro} belongs to $L^\infty_{\fil}(0,T;L^\infty(\dom))$ and satisfies
\begin{equation}\label{eq:linf_est}
\esssup_{(\omega,t)\in\Omega\times[0,T]}\|y(t)\|_{L^\infty(\dom)} \leq \exp (C T ) \left( \norme{y_T}_{L^{\infty}_{\mathcal F_T}(\Omega;L^{\infty}(\mathcal O))} + \norme{F}_{L^{\infty}_{\fil}(0,T;L^{\infty}(\dom))}\right),
\end{equation}
for some $C>0$ only depending on $\dom$, $\alpha$ and $\beta$.
\end{theo}

Theorems \ref{thm:ito} and \ref{thm:infinity} should be compared with \cite[Theorem 3.2]{HMY02}, where similar results are presented for one-dimensional BSPDEs. However, our contributions differ in several ways. Firstly, while \cite{HMY02} is restricted to the one-dimensional case, our results are established for BSPDEs in any dimension. {Secondly, we address weak solutions rather than strong solutions. This distinction is important since weak solutions are less regular in the spatial variable and cannot be evaluated pointwise, preventing the direct application of Itô's formula to the function $y^{2p}$ as in \cite{HMY02}. Instead, we establish an Ito's formula for the $L^p$-norm in the backward case (see \Cref{prop:ito} below), using some ideas from \cite{dareiotis2015} for SPDEs, which involve {direct and constructive arguments}, like truncations and limit procedures for a suitable approximation of the $L^p$-norm (see Sections \ref{sec_tech} and \ref{sec_ito_proof})}. Finally, our work considers coefficients that are merely bounded, as opposed to some differentiability conditions imposed in \cite[eq. (2.4)]{HMY02}. Similarly, we relax the conditions on the initial datum and the source term, which are less restrictive compared to those required by \cite[eqs. (2.2)--(2.3)]{HMY02}.

As in \cite{HMY02}, we also obtain a uniform bound on the first component of the solution of the BSPDE (see Theorem \ref{thm:infinity}), which initially can be surprising due to the presence of the stochastic integral. Nevertheless, as mentioned before, the noise term in \eqref{eq:intro} is endogenous to the equation, and as noted in \cite{DT12}, this special feature makes the theory for these equations closer to deterministic PDEs than to SPDEs. 

{
\begin{rmk}
We also note that our approach provides an improved dependence on the variable
$\omega\in\Omega$ for the process $Y$; see \cite[Theorem~3.2]{HMY02} together with
\Cref{thm:ito}.  
In contrast, within the $L^\infty$ framework of \Cref{thm:infinity} no improvement in
the regularity of $Y$ is obtained.  
We refer to \Cref{new_sec_reg} for additional discussion related to the potential optimality of these results.
\end{rmk}
}

\begin{rmk}
The proof of \Cref{thm:infinity} is obtained by refining some of the estimates derived in the proof of \Cref{thm:ito}. In particular, we establish a version of \eqref{est_lp} with constants that are uniform in $p$, and then pass to the limit as $p \to \infty$.

Alternatively, one might consider using comparison principle arguments to derive the estimate in \Cref{thm:infinity}; see, for instance, \cite[Theorem 5.1]{DT12}. That result is stated under assumptions that guarantee the existence of strong solutions and require some differentiability of the initial data. Whether these conditions can be relaxed to cover our setting remains unclear. We therefore choose to follow a different, more quantitative route, which stays within the weak solution framework and naturally yields the $L^p$-estimates needed in our applications below.
\end{rmk}

\subsection{Applications}\label{app}

Once we have established the regularity of the BSPDE \eqref{eq:intro}, we are interested in studying the following applications:
\begin{enumerate}
	\item Local existence of semilinear equations.
	\item {Controllability for linear and semilinear systems.}
\end{enumerate}

The first application is concerned about the local existence of solutions to the semilinear system 
\begin{equation}
\label{eq:semi}
\begin{cases}
\d{y}=-\left(\Delta y + \alpha y + \beta Y + f(y) \right)\d{t}+Y\d{W}(t) &\text{in } (0,T)\times\dom, \\
y=0 &\text{on } (0,T)\times\partial\dom, \\
y(T)=y_T &\text{in }\dom,
\end{cases}
\end{equation}
where we assume that
\begin{enumerate}[label={\textnormal{\bf (F)}}]
\item \label{fn} $f\in C^\infty(\mathbb R)$ such that $f(0)=0$.
\end{enumerate}

A pair of random fields $(y,Y)$ is called a weak solution to \eqref{eq:semi} if
\begin{enumerate}
\item[1)] $(y,Y)$ is $L^2(\mathcal O)\times L^2(\mathcal O)$-valued and $\mathcal F_t$-measurable for each $t\in[0,T]$,
\item[2)] $(y,Y)\in \left[L^2_{\fil}(\Omega;C([0,T]);L^2(\dom))\bigcap L^2_{\fil}(0,T;H_0^1(\dom))\right]\times L^2_{\fil}(0,T;L^2(\dom))$, and
\item[3)] for any $t\in[0,T]$ and $\phi\in H_0^1(\mathcal O)$ it holds
\begin{align*}
(y(t),\phi)_{L^2(\mathcal O)}&=(y_T,\phi)_{L^2(\mathcal O)}-\int_{t}^{T}(\nabla y,\nabla \phi)_{L^2(\mathcal O)}\d{s}\\
&\quad +\int_{t}^{T}(\alpha y+\beta Y+f(y),\phi)_{L^2(\mathcal O)}\d{s}-\int_{t}^{T}(Y,\phi)_{L^2(\mathcal O)}\d{W}(s), \quad \textnormal{a.s.}
\end{align*}
\end{enumerate}

To state our result, let us define the functional space $\mathscr X:=L^2_{\fil}(\Omega;C([0,T;]L^2(\dom)))\cap L^2_{\fil}(\Omega;L^2(0,T;H_0^1(\dom)))\cap L^\infty_{\fil}(0,T;L^\infty(\dom))$ and for any $p\in[2,+\infty)$ we set 
\begin{equation}\label{eq:Yp_space}
\mathscr Y_p :=\left\{(y,Y)\in \mathscr X \times   L^{p}_{\fil}(\Omega;L^2(0,T;L^2(\dom)))\right\}
\end{equation}
endowed with its natural norm, that is,
$
\|(y,Y)\|_{\mathscr Y_p}=\|y\|_{\mathscr X}+\|Y\|_{L^{p}_{\fil}(\Omega;L^2(0,T;L^2(\dom)))}.
$
%

\begin{theo}\label{thm:semi}
Let $T>0$, $p\in[2,+\infty)$ and assume that \ref{fn} holds. There exists $\delta>0$ only depending on $\dom$, $T$, $p$ and $f$ such that for any $y_0\in L^\infty_{\F_T}(\Omega;L^\infty(\dom))$ satisfying $\|y_T\|_{L^\infty_{\F_T}(\Omega;L^\infty(\dom))}\leq\delta$, there is a unique weak solution $(y,Y)$ to \eqref{eq:semi} that belongs to $ \mathscr Y_p$.
\end{theo}

The proof of \Cref{thm:semi} relies on a standard Banach fixed-point procedure and exploits the boundedness of the solution provided by \Cref{thm:infinity} to prove the contractivity of a suitable nonlinear map.

{Note that we do not impose any growth or sign conditions on the nonlinear term.  
Thanks to the higher regularity already established for the solution, we can apply a suitable Taylor expansion to the nonlinear term, which allows us to bypass the usual difficulties associated with low regularity (see  \Cref{sec:semilinear} and \Cref{lem:nonlinearity}).  
In this way, we obtain a local well-posedness result for \eqref{eq:semi} for sufficiently small terminal data.  
This avoids the standard global Lipschitz assumptions commonly used in the literature (see, e.g., \cite{HP91,HMY02} or \cite[Chapter~4]{LZ21}) and provides a more flexible framework for handling semilinear systems.}

Our second application is concerned about the controllability of \eqref{eq:semi}. Controllability is a qualitative property of dynamical systems that refers to the ability to steer the system from any initial state to any desired final state within a finite time period, using an appropriate control input. {To fix ideas, let us consider the linear control system}
 \begin{equation}
\label{eq:intro_control}
\begin{cases}
\d{y}=-\left(\Delta y + \alpha y + \beta Y + \chi_{\dom_0}h\right)\d{t}+Y\d{W}(t) &\text{in } (0,T)\times\dom, \\
y=0 &\text{on } (0,T)\times\partial\dom, \\
y(T)=y_T &\text{in }\dom,
\end{cases}
\end{equation}
where $h$ is an external control force localized on a set $\dom_0\subset \dom$. In this part we assume that $\mathcal O$ is connected. 

Due to Theorems \ref{thm:l2_sol}, \ref{thm:ito} and \ref{thm:infinity}, system \eqref{eq:intro_control} is well-posed for any terminal datum $y_T\in L^p_{\F_T}(\Omega;L^p(\dom))$ and any function $h\in L^p_{\fil}(0,T;L^p(\dom_0))$ with $p\in[2,+\infty]$. This motivates us to find controls $h$ such that the following controllability condition is fulfilled. 

\begin{defi}
\label{def:controlLp}
Let $p \in [2, \infty]$. System \eqref{eq:intro_control} is said to be null-controllable in $L^p$ if for any terminal datum $y_T\in L^p_{\F_T}(\Omega;L^p(\dom))$, there exists a control $h\in L^p_{\fil}(0,T;L^p(\dom_0))$ such that the corresponding weak solution $(y,Y)$ of \eqref{eq:intro_control} satisfies
\begin{equation*}
y(0,\cdot)=0 \quad\textnormal{a.s.}
\end{equation*}
\end{defi}
In the case $p=2$, this problem has been studied extensively, most notably in \cite{BRT03} and \cite{TZ09}. Using duality arguments, the controllability for \eqref{eq:intro_control} is established via Carleman estimates, which have been adapted from their classical deterministic form (see \cite{fursi}) to the stochastic setting. However, to our knowledge, there are no results in the literature addressing the case when $p>2$. To bridge this gap, we present a result covering the full range $p\in(2,+\infty]$ for a simplified yet interesting case. 

\begin{proposition}\label{thm:control}
Let $p\in(2,+\infty]$ and assume that $\beta \equiv 0$.  Then, system \eqref{eq:intro_control} is null-controllable in $L^p$. Moreover, if $T\in(0,1)$, the following estimate on the control holds
\begin{equation}
\label{eq:EsimationControlhLp}
\|h\|_{L^{p}_{\fil}(0,T;L^p(\dom))} \leq \exp(C/T) \|y_T\|_{L^{p}_{\F_T}(\Omega;L^p(\mathcal O))},
\end{equation}
for some $C>0$ depending at most on $\dom$, $\dom_0$, $p$ and $\alpha$.
\end{proposition}

{By simplifying the original system, we employ the duality approach used in \cite[Section 7]{TZ09},  but following the spirit of \cite{Liu14}, where we avoid the need to prove stochastic Carleman estimates and use instead known results in the deterministic setting. Our result extends those in \cite{BRT03} and \cite{TZ09} to a wider range of values of $p$, circumventing technical challenges typically associated with stochastic control systems. }

{
In this spirit, our final result brings together both applications and aims to establish a controllability result for the following controlled backward stochastic semilinear heat equation
\begin{equation}
\label{eq:backward_nonlinear}
\begin{cases}
\d{y} = \left(-\Delta y + f(y) + \chi_{\dom_0}h\right)\d{t} + Y\d{W}(t) &\text{in } (0,T)\times\dom, \\
y = 0 &\text{on } (0,T)\times \partial\dom, \\
y(T) = y_T &\text{in } \dom.
\end{cases}
\end{equation}
In the $L^2$-setting, the \rouge{(global)} controllability of \eqref{eq:backward_nonlinear} in the sense of \Cref{def:controlLp} has been studied in \cite{HSLBP23} (see also \cite{ZXL24}), under the assumption that $f$ is globally Lipschitz. However, as in the well-posedness results discussed earlier, this imposes a restrictive condition on the growth of $f$, which limits the applicability of the result.

A natural way to go beyond this limitation is to weaken the notion of controllability. In this direction, and in order to accommodate nonlinearities with weaker growth assumptions, we adopt the following notion of local controllability, which will serve as the framework for our result.

\begin{defi}[Local controllability]\label{def:local_control}
We say that equation \eqref{eq:backward_nonlinear} is \emph{locally controllable} in $L^\infty$ if there exists $\delta > 0$ such that for every terminal condition $y_T \in L^{\infty}_{\F_T}(\Omega;L^{\infty}(\mathcal O))$ with $\|y_T\|_{L^{\infty}_{\F_T}(\Omega;L^{\infty}(\mathcal O))} \leq \delta$, there exists a control $h \in L^{\infty}_{\fil}(0,T;L^\infty(\mathcal O_0))$ such that the weak solution $(y,Y)$ of \eqref{eq:backward_nonlinear} satisfies $y(0) = 0$ a.s.
\end{defi}

Using the control result in the linear setting stated in \Cref{thm:control} for $p=+\infty$ and in view of \Cref{thm:semi}, we can state the following result.
\begin{theo}\label{theo:control_semi}
There is $\delta>0$ depending only on $T$, $\dom$, and $\dom_0$ such that for every $y_T \in L^{\infty}_{\F_T}(\Omega;L^{\infty}(\mathcal O))$ satisfying $\|y_T\|_{L^{\infty}_{\F_T}(\Omega;L^{\infty}(\mathcal O))} \leq \delta$, there exists a control $h \in L^{\infty}_{\fil}(0,T;L^\infty(\mathcal O_0))$ such that the weak solution $(y,Y)$ of \eqref{eq:backward_nonlinear} satisfies $y(0) = 0$ a.s.
\end{theo}
{
The proof of this result combines several techniques. First, we establish a null-controllability result for the linear backward equation with an $L^{\infty}$-weighted source term, crucially relying on the cost estimate \eqref{eq:EsimationControlhLp}. To achieve this, we adapt ideas from \cite{LLT13} in the deterministic framework and from \cite{HSLBP20a} in the stochastic setting. For completeness, we present this intermediate result in our context here. The final step of the proof employs a Banach fixed-point argument, where arguments used in \Cref{thm:semi} play a key role in controlling the nonlinear terms.}

\begin{rmk}
Theorem \ref{theo:control_semi} provides, to the best of our knowledge, the first true local controllability result for backward semilinear SPDEs. Previous works have not directly addressed this question. The closest related result appears in Section 4 of \cite{HSLBP20a}, where the authors establish a weaker notion of controllability, {named statisical local controllability: instead of reaching $y(0)=0$ a.s., they prove that $\mathbb{P}(y(0)=0) \geq 1 - \varepsilon$ for some $\varepsilon>0$ sufficiently small starting from a small initial data.} This distinction is significant, as our result guarantees local controllability in the almost sure sense.

It is also worth noting that \cite{HSLBP20a} considers the full backward equation, including the $\beta Y$ term (compare \eqref{eq:intro_control} and \eqref{eq:backward_nonlinear}). In contrast, our approach does not treat this term because of the limitations in the linear controllability analysis (see \Cref{thm:control}).  Nevertheless, it does establish a controllability result in a setting where previously only high-probability statements were known.
\end{rmk}
}

{
\begin{rmk}
\Cref{theo:control_semi} does not cover the case in which the source term $f$ depends on $Y$.  
Indeed, in the key a priori estimate of the source-term method in \Cref{prop:source}, the source term is required to lie in $L^\infty$, while our analysis only yields $L^p$-regularity for the process $Y$ (with finite $p$), which is generally insufficient to close the estimates.  
This obstruction already appears at the level of $L^2$, as noted in our previous work (see \cite[Remark~4.4]{HSLBP20a}), where the source-term method does not provide suitable bounds when $f$ depends on $Y$.  
Addressing such nonlinear dependencies would likely require methods of a different nature, and we view this as an interesting direction for future research.
\end{rmk}
}

{
\begin{rmk}
For forward semilinear SPDEs of the form
\begin{equation}\label{eq:forward_intro}
\begin{cases}
\d y = (\Delta y + \chi_{\dom_0}h + f(y))\,\dt + (a y + g(y))\,\d W(t) & \text{in } (0,T)\times\dom,\\
y=0 & \text{on } (0,T)\times\partial\dom,\\
y(0)=y_0 & \text{in }\dom,
\end{cases}
\end{equation}
the basic controllability question is whether one can choose a suitable control $h$ (and possibly an additional control acting in the diffusion) to ensure
\[
y(T)=0 \quad \text{a.s.}
\]

Let us recall that the linear case (i.e.\ $f\equiv 0$, $g\equiv 0$) has been studied extensively.  
Several controllability results are available under different assumptions, relying on Carleman estimates (see \cite{fursi}) and on the Lebeau--Robbiano strategy \cite{LR95}, with both approaches requiring suitable adaptations to the stochastic setting  
(see, e.g., \cite{BRT03,TZ09,LU11,Liu14}).  
We also point out that, in some of these works, an additional control acting in the diffusion term is required; namely, the term $a y\,\d W(t)$ in \eqref{eq:forward_intro} is replaced by $(a y + H)\,\d W(t)$ for a suitably chosen control~$H$.

When nonlinearities are present, particularly in the diffusion term, the analysis becomes substantially more delicate.  
In the globally Lipschitz setting, our previous work \cite{HSLBP23} (see also \cite{ZXL24}) established global controllability but---as in the linear case---only by introducing an additional control in the diffusion term.  
Subsequently, in \cite{HSLBP20a}, we proved statistical local controllability with a single control (that is, $\mathbb{P}(y(T)=0)\geq 1-\varepsilon$ for $\varepsilon>0$ small and small initial data), under suitable structural assumptions on the nonlinear diffusion term $g(y)$, essentially in one spatial dimension.  
However, the general problem of proving {true} local controllability for semilinear equations with nonlinear noise remains open.

The present paper suggests a possible direction for improving forward results: the high-regularity estimates we obtain here for the backward equation appear to be precisely the type of regularity input that is currently missing in the forward analysis.  
Developing analogous high-regularity estimates for forward equations could plausibly lead to sharper controllability results, but this remains unexplored.
\end{rmk}
}

{
\begin{rmk}
The $L^p$- and $L^\infty$-estimates obtained in this work might also be relevant to stochastic optimal control for parabolic SPDEs.  
Backward equations naturally arise in Pontryagin-type maximum principles, and high-integrability or boundedness of the adjoint variables might play an important role in problems with bounded controls or state constraints.  
Although we do not pursue this direction here, the regularity tools developed in the present paper could extend to the optimal control setting.  
We refer the interested reader to \cite{FHT12,DM13,LZ14,FL20} and the references therein for further developments in this direction.  
We hope that the present results may provide a useful starting point for future work, and we leave this as an interesting avenue for further investigation.
\end{rmk}
}

%

\medskip

\subsection{Outline of the paper}

The rest of the paper is organized as follows. In \Cref{sec:ito}, we present and prove an Itô formula for the $L^p$-norm of the solution of the BSPDE \eqref{eq:intro}. As a consequence, we will deduce the proof of \Cref{thm:ito}. Then, in \Cref{sec:estimates} we use the Itô formula to show uniform $L^p$-estimates then \Cref{thm:infinity} concerning $L^\infty$-estimates for the solution of the BSPDE \eqref{eq:intro}. \Cref{sec:semilinear} focuses on the study of the semilinear equation \eqref{eq:semi} and its well-posedness, while \Cref{sec:control} is devoted to prove the control results in \Cref{thm:control} and  \Cref{theo:control_semi}. {Lastly, in \Cref{new_sec_reg} we give further remarks on the regularity obtained in Theorems \ref{thm:ito} and \ref{thm:infinity}.}

\section{It\^o's formula for the $L^{p}$-norm}\label{sec:ito}

The goal of this section is to state Itô's formula for the $L^{p}$-norm and deduce $L^p$-estimates for the weak solution $(y,Y)$ of \eqref{eq:intro}. Note that throughout this section, the $L^p$-estimates are not uniform with respect to $p$. The final conclusion will be the proof of \Cref{thm:ito}.

We begin by proving the following result, which is essential for the proofs of \Cref{thm:ito} and \Cref{thm:infinity}. It provides an initial a priori estimate together with an Itô's formula for the $L^p$-norm of the solution to the backward SPDE \eqref{eq:intro}. More precisely, we have the following.

\begin{proposition}[Ito's formula for $L^p$-norm]\label{prop:ito}
Let $(y,Y)$ be a weak solution to \eqref{eq:intro} and assume that $y_T\in L^p_{\mathcal F_T}(\Omega;L^p(\dom))$ and $F\in L^{p}_{\fil}(0,T;L^p(\mathcal O))$ for some $p\in[2,+\infty)$. Then, there is a constant $C>0$ only depending on $\dom$, $\alpha$, $\beta$, and $p$ such that
\begin{align}\notag
&\E\left(\sup_{t\in[0,T]}\|y(t)\|^p_{L^p(\dom)}\right)+\E\left(\int_0^T\int_{\dom}|y|^{p-2} |\nabla y|^2\dx \d t\right)\\ \label{avoc}
&+\E\left(\int_0^T\int_{\dom}|y|^{p-2}Y^2\dx \dt \right)\leq Ce^{C T}\E\left(\int_{\dom}|y_T|^p\dx+\int_0^T\int_{\dom}|F|^p\dx \dt\right).
\end{align}
Moreover, for any $t\in[0,T]$, we have
\begin{align}\notag
&\int_{\dom}|y(t)|^p\dx+p(p-1)\int_t^T\int_{\dom}|y|^{p-2}|\nabla y|^2\dx\d s+\frac{p(p-1)}{2}\int_t^T\int_{\dom}|y|^{p-2}|Y|^2\dx\d s\\ \label{itosfor}
&=\int_{\dom}|y_T|^p\dx+p\int_t^T\int_{\dom}|y|^{p-2}y\left(\alpha y+\beta Y+F\right)\dx\d{s}-p\int_t^T\int_{\dom}|y|^{p-2}yY\d{x}dW(s), \quad \textnormal{a.s.}
\end{align}
\end{proposition}
%

%
%
%
%
%

For the proof, we follow the methodology in \cite{dareiotis2015} adapted to the backward case and for the reader's convenience we present it in several steps. 

\subsection{Technical lemmas}\label{sec_tech}

We begin by presenting some auxiliary functions. For each $n\in\mathbb N$ and some $p\geq 2$, consider the twice continuously differentiable function $\phi_n:\mathbb R\to\mathbb R$ defined by 
\begin{align}\label{apr}
\phi_n(r):=
\begin{cases}
    |r|^p  & \text{ if } |r|<n, \\
     n^{p-2}\frac{p(p-1)}{2}\left(|r|-n\right)^2+pn^{p-1}\left(|r|-n\right)+n^p & \text{ if } |r|\geq n.
\end{cases}
\end{align}
It is not difficult to check that $\phi_n$ satisfy
\begin{equation}\label{eq:quad_bound}
|\phi_n(r)|\leq M|r|^2, \qquad |\phi_n^\prime(r)|\leq M|r|, \qquad |\phi^{\prime\prime}(r)|\leq M,\qquad r \in \R,
\end{equation}
where $M>0$ depends only on $p$ and $n$. Furthermore, we have
\begin{equation}\label{eq:p_bounds}
|\phi_n(r)|\leq N|r|^p, \qquad |\phi_n^\prime(r)|\leq N|r|^{p-1}, \qquad |\phi^{\prime\prime}(r)|\leq N|r|^{p-2},\qquad r \in \R,
\end{equation}
for a positive constant $N$ only depending on $p$, and if $n\to \infty$, we have
\begin{align}\label{prop:convr}
\phi_n(r)\to |r|^p,\qquad  \phi_n^{\prime}(r)\to p|r|^{p-2}r, \qquad \phi_n^{\prime\prime}(r)\to p(p-1)|r|^{p-2},\qquad r \in \R.
\end{align}
Additionally, we can check that the following inequalities hold
\begin{align}\label{prop_r_phi_n}
|r\phi_n^\prime(r)| &\leq p\phi_n(r), \\ \label{prop_phi_n}
|\phi_n^\prime(r)|^2 &\leq 4 p \phi_n^{\prime\prime}(r)\phi_n(r), \\ \label{prop_phi_pp_n}
[\phi^{\prime\prime}(r)]^{p/(p-2)} &\leq \left[p(p-1)\right]^{p/(p-2)}\phi_n(r)\qquad \textnormal{for all } r \in \R,
\end{align}
and, by the definition of \eqref{apr}, we can see that 
\begin{equation}\label{eq:pos_phi_pp}
\phi_n^{\prime\prime}(r)\geq 0 \quad \textnormal{for all $r\in\mathbb R$}.
\end{equation}

We start with the following lemma.
\begin{lemma}\label{pf:mart}
Let $n\in \mathbb{N}$, $p\in[2+\infty)$, $t\in[0,T]$, and $(y,Y)$ be a weak solution to \eqref{eq:intro}. Then
\begin{equation*}
\E\left(\int_t^T\int_{\dom}\phi_n^{\prime}(y)Y\dx\,\d W(s)\right)=0.
\end{equation*}
\end{lemma}

\begin{proof}
By Cauchy-Schwarz and H\"older inequalities together with \eqref{eq:quad_bound} we see that 
\begin{align*}
\int_t^T\left(\int_{\dom}\phi_n^{\prime}(y)Y\dx\right)^2\d s&\leq \int_t^T \|\phi_n^{\prime}(y)\|^2_{L^2(\dom)}\|Y\|^2_{L^2(\dom)}\d s\\
&\leq  M \int_t^T\|y\|^{2}_{L^2(\dom)}\|Y\|^2_{L^2(\dom)}\d s\\
&\leq  M \sup_{s\in(t,T]}\|y(s)\|^{2}_{L^2(\dom)}\int_t^T\|Y\|^2_{L^2(\dom)}\d s, \quad \text{a.s.}
\end{align*}
for any $t\in[0,T]$. Taking expectation and using again Cauchy-Schwarz's inequality we get
\begin{align*}
\E\left(\left[\int_t^T\left(\int_{\dom}\phi_n^{\prime}(y)Y\dx\right)^2\d s\right]^{1/2}\right)\leq M \E\left(\sup_{s\in(t,T]}\|y(s)\|^{2}_{L^2(\dom)}\right)^{1/2}\E\left(\int_t^T\|Y\|^2_{L^2(\dom)}\d s\right)^{1/2}
\end{align*}
whence 
\begin{align*}
&\E\left(\left[\int_t^T\left(\int_{\dom}\phi_n^{\prime}(y)Y\dx\right)^2\d s\right]^{1/2}\right)\\
&\leq\frac{M}{2}\left[\E\left(\sup_{s\in(t,T]}\|y(s)\|^{2}_{L^2(\dom)}\right)+\E\left(\int_t^T\|Y\|^2_{L^2(\dom)}\d s\right)\right]< +\infty.
\end{align*}
by the regularity estimate \eqref{thm:l2_reg}. Defining $H(s):=\int_{\dom}\phi_n^\prime(y(s))Y(s)\dx$, we have proved that $H(\cdot)\in L^2(\Omega\times(0,T))$. Therefore, $\E\left(\int_t^{T}H(s)\d{W}(s)\right)=\E\left(\int_0^{T}H(s)\d{W}(s)-\int_0^t H(s)\d{W}(s)\right)=0$. This ends the proof.
\end{proof}

We continue with the following identity.
\begin{lemma}\label{lem:ipp}
Let $n\in \mathbb{N}$ and $(y,Y)$ be a weak solution to \eqref{eq:intro}. For any $t\in[0,T]$, the following identity holds 
\begin{align}\notag
\int_{\dom}&\phi_n(y(t))\dx+\int_t^T\int_{\dom}\phi^{\prime\prime}_n(y)|\nabla y|^2\dx\d s+\frac{1}{2}\int_t^T\int_{\dom} \phi^{\prime}_n(y)Y^2 \dx\d s\\ \label{aft_ito}
&=\int_{\dom}\phi_n(y(T))\dx+ \int_t^T\int_{\dom} \phi^{\prime}_n(y)(\alpha y+\beta Y+F) \dx\d s - \int_t^T\int_{\dom}\phi_n^{\prime}(y)Y\dx\d W(s), \quad \text{ a.s.}
\end{align}
\end{lemma}

\begin{proof}
Since $\phi_n$ is a twice continuously differentiable function, we can apply It\^o's formula and use equation \eqref{eq:intro} to get
\begin{align*}\notag
\int_{\dom}\phi_n(y(T))\dx&=\int_{\dom}\phi_n(y(t))\dx-\int_t^T\int_{\dom}\phi^{\prime}_n(y)\Delta y\dx\d{s}-\int_t^T\int_{\dom}\left(\phi^{\prime}_n(y)F-\frac{1}{2}\phi^{''}_n(y)Y^2\right)\dx\d s \\ 
&-\int_{t}^{T}\int_{\dom} \phi_n^\prime(y)(\alpha y +\beta Y)\dx\d{s}+\int_t^T\int_{\dom}\phi_n^{'}(y)Y\dx\d W(s), \quad \text{ a.s.}
\end{align*}
for any $t\in[0,T]$. Note that for $i=1,\ldots,d$, $\partial_{x_i}(\phi_n^\prime(y))=\phi_n^{\prime\prime}(y)\partial_{x_i}y$, whence integrating by parts in the space variable and using the homogeneous Dirichlet boundary condition of \eqref{eq:intro} yields 
\begin{align*}\notag
\int_{\dom}\phi_n(y(T))\dx&=\int_{\dom}\phi_n(y(t))\dx+\int_t^T\int_{\dom}\phi^{\prime\prime}_n(y)|\nabla y|^2\dx\d s-\int_t^T\int_{\dom}\left(\phi^{\prime}_n(y)F-\frac{1}{2}\phi^{\prime\prime}_n(y)Y^2\right)\dx\d s\\ 
&-\int_{t}^{T}\int_{\dom} \phi_n^\prime(y)(\alpha y +\beta Y)\dx\d{s}+\int_t^T\int_{\dom}\phi_n^{\prime}(y)Y\dx\d W(s), \quad \text{ a.s.}
\end{align*}
After rearranging some terms, we obtain the desired equality.
\end{proof}

In the remainder of this section, $C$ denotes a generic positive constant uniform with respect to $n$ and that may change from line to line. We have the following estimate.

\begin{lemma}\label{prop_1}
Let $n\in \mathbb{N}$, $p\in[2,+\infty)$, $(y,Y)$ be a weak solution to \eqref{eq:intro} and assume that $y_T\in L^p_{\mathcal F_T}(\Omega;L^p(\dom))$ and $F\in L^{p}_{\fil}(0,T;L^p(\mathcal O))$. Then, the following estimate holds
\begin{align}\label{est:2}
\E&\left(\int_0^T\int_{\dom}\phi_n^{\prime\prime}(y)|\nabla y|^2\dx\d s\right)+\E\left(\int_0^T\int_{\dom}\phi_n^{\prime\prime}(y)Y^2\dx\d s\right)\\\notag
&\leq C e^{CT}\E\left(\int_{\dom}|y_T|^p\dx+\int_0^T\int_{\dom}|F|^p\dx\d s \right),
\end{align}
for some constant $C>0$ that only depends on $p$, $\|\alpha\|_{L^\infty_{\fil}(0,T;L^\infty(\dom))}$ and $\|\beta\|_{L^\infty_{\fil}(0,T;L^\infty(\dom))}$.
\end{lemma}

\begin{proof}
Let $t>0$ and $0<r\leq t$. From identity \eqref{aft_ito} of \Cref{lem:ipp}, we take conditional expectation to obtain
\begin{align*}\notag
\E\left(\int_{\dom}\right.&\left.\phi_n(y(t))\dx+\int_t^T\int_{\dom}\phi^{\prime\prime}_n(y)|\nabla y|^2\dx\d s+\frac{1}{2}\int_t^T\int_{\dom}\phi^{\prime\prime}_n(y)Y^2\dx\d s \bigmid \F_r \right)\\\notag
&=\E\left(\int_{\dom}\phi_n(y(T))\dx+\int_t^T\int_{\dom}\phi^{\prime}_n(y)(\alpha y+\beta Y+F)\dx\d s  \bigmid \F_r\right)
\end{align*}
where we have used that $\E\left(\int_t^T\int_{\dom}\phi_n^{\prime}(y)Y\dx\d W(s) \mid \F_r\right)=0$ due to \Cref{pf:mart}. 
From  \eqref{prop_r_phi_n}-\eqref{prop_phi_pp_n}, the positivity property \eqref{eq:pos_phi_pp} and Young's inequality, we can deduce
\begin{align}\label{prop:1}
|\phi_n^\prime(y)F| &\leq C|F| \phi_n(y)^{(p-1)/p}\leq C\big(|F|^p+\phi_n(y)\big), \\ \label{prop:2}
|\phi_n^\prime(y)\alpha y| &\leq C|\alpha| \phi_n(y), \\ \label{prop:3}
|\phi_n^\prime(y)\beta Y| &\leq \delta \phi_n^{\prime\prime}(y)|Y|^2+\tfrac{C|\beta|^2}{\delta}\phi_n(y),
\end{align}
almost surely, for a constant $C>0$ only depending on $p$ and any $\delta \in(0,1)$. Therefore, taking $\delta$ small enough
\begin{align*}
\E\left(\int_{\dom}\right.&\left.\phi_n(y(t))\dx+\int_t^T\int_{\dom}\phi^{\prime\prime}_n(y)|\nabla y|^2\dx\d s+\int_t^T\int_{\dom}\phi^{\prime\prime}_n(y)Y^2\dx\d s \bigmid \F_r \right)\\\notag
&\leq C\E\left(\int_{\dom}\phi_n(y(T))\dx+\int_t^T\int_{\dom}|F|^p\dx\d s+\int_t^T\int_{\dom}\phi_n(y)\dx\d s \bigmid \F_r\right),
\end{align*}
where $C>0$ depends on $p$ and the norms $\|\alpha\|_{L^\infty_{\fil}(0,T;L^\infty(\dom))}$ and $\|\beta\|_{L^\infty_{\fil}(0,T;L^\infty(\dom))}$.

Taking expectation in the above inequality and by Fubini's theorem and \Cref{pf:mart}, we get
\begin{align}\notag
\E\left(\int_{\dom}\right.&\left.\phi_n(y(t))\dx+\int_t^T\int_{\dom}\phi^{\prime\prime}_n(y)|\nabla y|^2\dx\d s+\int_t^T\int_{\dom}\phi^{\prime\prime}_n(y)Y^2\dx\d s \right)\\ \label{est:inter}
&\leq C\E\left(\int_{\dom}\phi_n(y(T))\dx+\int_0^T\int_{\dom}|F|^p\dx\d s\right)+C\int_t^T \E \left(\int_{\dom}\phi_n(y)\dx\d s  \right).
\end{align}

By \eqref{eq:pos_phi_pp}, we can drop the second and third terms in the above estimate and, by the backward Grönwall inequality (see \Cref{lem:back_gron}), we obtain
\begin{align}\label{est:inter2}
\E\left(\int_{\dom}\phi_n(y(t))\dx \right) 
\leq Ce^{CT}\E\left(\int_{\dom}\phi_n(y(T))\dx+\int_0^T\int_{\dom}|F|^p\dx\d s\right).
\end{align}
Integrating in time \eqref{est:inter2} and combining with \eqref{est:inter} entails 
\begin{align}\notag
\E&\left(\int_{\dom}\phi_n(y(t))\dx\right)+\E\left(\int_t^T\int_{\dom}\phi_n^{\prime\prime}(y)|\nabla y|^2\dx\d s\right)+\E\left(\int_t^T\int_{\dom}\phi_n^{\prime\prime}(y)Y^2\dx\d s\right)\\ \label{est:inter_f}
&\leq Ce^{CT}\left[\E\left(\int_{\dom}\phi_n(y(T))\dx\right)+\E\left(\int_0^T\int_{\dom}|F|^p\dx\d s \right)\right],
\end{align}
where we have replaced the terminal data of \eqref{eq:intro}. Dropping the first term in \eqref{est:inter_f} and applying the monotone convergence theorem yield the desired result. 
\end{proof}

The following result tells us that we can improve \Cref{prop_1} by adding a supremum (in time) estimate on the left-hand side of \eqref{est:2}.

\begin{lemma}\label{prop_supremum}
Under the assumptions of \Cref{prop_1}, there is $C>0$ only depending on $p$ such that 
\begin{align}\label{est:2_mejorado}
&\E\left(\sup_{t\in[0,T]}\int_{\mathcal O}\phi_n(y(t))\right)+\E\left(\int_0^T\int_{\dom}\phi_n^{\prime\prime}(y)|\nabla y|^2\dx\d s\right)+\E\left(\int_0^T\int_{\dom}\phi_n^{\prime\prime}(y)Y^2\dx\d s\right)\\\notag
&\quad \leq C e^{CT}\E\left(\int_{\dom}|y_T|^p\dx+\int_0^T\int_{\dom}|F|^p\dx\d s \right).
\end{align}
\end{lemma}

\begin{proof}
Using as a starting point identity \eqref{aft_ito} of \Cref{lem:ipp}, we use \eqref{prop:1}--\eqref{prop:3} to deduce that almost surely
\begin{align*}
\int_{\dom}&\phi_n(y(t))\dx+\int_t^T\int_{\dom}\phi^{\prime\prime}_n(y)|\nabla y|^2\dx\d s+\int_t^T\int_{\dom}\phi^{\prime\prime}_n(y)Y^2\dx\d{s}\\
&\leq C\left(\int_{\dom}|y_T|^p\dx+\int_t^{T}\int_{\dom}|F|^p\dx \d s+\int_t^{T}\int_{\dom}\phi_n(y)\dx\d s\right)-\int_t^T\int_{\dom}\phi_n^{\prime}(y)Y\dx\d W(s),
\end{align*}
for any $t\in[0,T]$ and some $C>0$ only depending on $p$, $\|\alpha\|_{L^\infty_{\fil}(0,T;L^\infty(\dom))}$ and $\|\beta\|_{L^\infty_{\fil}(0,T;L^\infty(\dom))}$. Hence
\begin{align}\notag
&\int_{\dom}\phi_n(y(t))\dx \leq C\left(\int_{\dom}|y_T|^p\dx+\int_0^{T}\int_{\dom}|F|^p\dx \d s\right)+\sup_{t\in[0,T]}\left|\int_t^T\int_{\dom}\phi_n^{\prime}(y)Y\dx\d W(s)\right|\\\label{bef:gronw}
&+C\int_t^{T}\int_{\dom}\phi_n(y)\dx\d{s}, \quad \text{ a.s.}
\end{align}

From the backward Grönwall inequality (see \Cref{lem:back_gron}) and taking supremum in $[0,T]$, we have that a.s.
\begin{align}\notag
\sup_{t\in[0,T]}&\int_{\dom}\phi_n(y(t))\dx\\ \label{bef:BDG}
&\leq Ce^{CT}\left(\int_{\dom}|y_T|^p\dx+\int_0^T\int_{\dom}|F|^p\dx\d s+2\sup_{t\in[0,T]}\left|\int_0^t\int_{\dom}\phi_n^{\prime}(y)Y\dx dW(s)\right|\right),
\end{align}
where we have used that
\begin{align}\notag 
\sup_{t\in[0,T]}&\left|\int_t^T\int_{\dom}\phi_n^{\prime}(y)Y\dx dW(s)\right| \\ \notag 
&\leq \left|\int_0^{T}\int_{\mathcal O}\phi_n^\prime(y)Y\dx\d{W}(s)\right|+\sup_{t\in[0,T]} \left|\int_0^{t}\int_{\mathcal O}\phi_n^\prime(y)Y\dx\d{W}(s)\right| \\ \label{lylys_trick}
& \leq 2 \sup_{t\in[0,T]} \left|\int_0^{t}\int_{\mathcal O}\phi_n^\prime(y)Y\dx\d{W}(s)\right|.
\end{align}

We take expectation in \eqref{bef:BDG} and apply Burkholder-Davis-Gundy inequality to get
\begin{align}\notag
&\E\left(\sup_{t\in[0,T]}\int_{\dom}\phi_n(y(t))\dx\right)\\\label{aft:BDG}
&\leq Ce^{CT} \left[\E\left(\int_{\dom}|y_T|^p\dx+\int_0^T\int_{\dom}|F|^p\dx\d s\right)+\E\left(\left|\int_0^T\left(\int_{\dom}\phi_n^{\prime}(y)Y\dx\right)^2\d s\right|^{1/2}\right)\right].
\end{align}
To estimate the last term in the above inequality we note that from \eqref{prop_phi_n}
\begin{equation*}
|\phi_n^{\prime}(y)Y|\leq 2 |Y| \phi_n^{\prime\prime}(y) ^{1/2} \phi_n(y)^{1/2} \quad\text{a.s.}
\end{equation*}
and applying H\"older and Young inequalities we have that there exists a numerical constant $C>0$ such that
\begin{align}\notag
&\E\left(\left|\int_0^T\left(\int_{\dom}\phi_n^{\prime}(y)Y\dx\right)^2\d s\right|^{1/2}\right)\\\label{aux:prp_1}
&\leq \delta \E\left(\sup_{s\in [0,T]}\int_{\dom}\phi_n(y)\dx\right)+\frac{C}{\delta}\E\left(\int_0^T\int_{\dom}|Y|^2|\phi_n^{\prime\prime}(y)|\dx\d s\right),
\end{align}
for any $\delta\in(0,1)$. Replacing \eqref{aux:prp_1} in \eqref{aft:BDG} and taking $\delta$ small enough we get
\begin{align}\notag
\E&\left(\sup_{t\in[0,T]}\int_{\dom}\phi_n(y(t))\dx\right)\\
&\leq Ce^{CT}\E \left(\int_{\dom}|y_T|^p\dx+\int_0^T\int_{\dom}|F|^p\dx\d s +  \int_0^T\int_{\dom}|Y|^2|\phi_n^{\prime\prime}(y)|\dx\d s \right). \label{aux:prop_2}
\end{align}
Finally, putting together estimate \eqref{est:2} of Lemma \ref{prop_1} and inequality \eqref{aux:prop_2} we obtain the desired result. This ends the proof.
\end{proof}

\subsection{Proof of the Ito's formula for the $L^p$-norm}\label{sec_ito_proof}

Now we are in position to prove the Ito's formula, that is, \Cref{prop:ito}.

\begin{proof}[Proof of \Cref{prop:ito}] Estimate \eqref{avoc}, follows from \eqref{prop:convr}, estimate \eqref{est:2_mejorado} in \Cref{prop_supremum} and a straightfoward application of Fatou's Lemma. This proves the first assertion of \Cref{thm:ito}.

To prove the It\^{o}'s formula \eqref{itosfor}, we proceed as follows. From \Cref{lem:ipp}, we have that the following identity holds
\begin{align}\notag
\int_{\dom}&\phi_n(y(t))\dx+\int_t^T\int_{\dom}\phi^{\prime\prime}_n(y)|\nabla y|^2\dx\d s + \frac{1}{2}\int_t^T\int_{\dom}\phi^{\prime\prime}_n(y)Y^2 \dx\d s\\ \label{ito_n}
& = \int_{\dom}\phi_n(y(T))\dx+\int_t^T\int_{\dom} \phi^{\prime}_n(y)(\alpha y+\beta Y+F) \dx\d s-\int_t^T\int_{\dom}\phi_n^{\prime}(y)Y\dx\d W(s) \quad \text{ a.s.,}
\end{align}
for all $t\in[0,T]$. We will see that each term in \eqref{ito_n} converges to the corresponding one in \eqref{itosfor} as $n\to \infty$. At this point, we will make use of estimate \eqref{avoc}.

From estimate \eqref{avoc}, we readily deduce that
\begin{equation}\label{eq:est_f_as}
\sup_{t\in[0,T]} \|y(t)\|^p_{L^p(\dom)}<+\infty, \quad \textnormal{a.s.}
\end{equation}
Then, by \eqref{eq:p_bounds}, \eqref{prop:convr}, \eqref{eq:est_f_as}, and dominated convergence theorem we obtain that 
\begin{align*}
\int_{\dom}\phi_n(y(t))\dx&\to \int_{\dom}|y(t)|^p\dx\quad \text{ a.s. for all $t\in[0,T]$}.
\end{align*}
This proves the convergence of the first terms in both sides of identity \eqref{ito_n}. 

Now we note from \eqref{eq:p_bounds} that a.s.
\begin{align*}
|\phi_n^{\prime}(y)F|\leq N|y|^{p-1}|F|
\end{align*} 
and, moreover, by succesive applications of H\"older inequality we have  for all $t\in[0,T]$ 
\begin{align*}
\int_t^T\int_{\dom} |y|^{p-1}|F|\dx\d s&\leq  \|y\|_{L^p((t,T];L^{p}(\dom))}^{p-1} \|F\|_{L^p((t,T];L^{p}(\dom))} < + \infty \quad \text{ a.s.},
\end{align*}
where the boundedness follows from \eqref{eq:est_f_as} and the regularity of $F$. Therefore, dominated convergence theorem together with \eqref{prop:convr} imply that
\begin{align*}
\int_t^T\int_{\dom}\phi_n^{\prime}(y)F\dx\d s\to p\int_t^T\int_{\dom}|y|^{p-2}yF\dx\d s \quad \text{ a.s. for all $t\in[0,T]$.}
\end{align*}

Similar arguments using the energy estimate \eqref{avoc} yield that
\begin{align*}
\int_t^T\int_{\dom}\phi_n^{\prime\prime}(y)|\nabla y|^2\dx\d s&\to p(p-1)\int_t^T\int_{\dom}|y|^{p-2}|\nabla y|^2\dx\d{s} \quad \text{ a.s.,} \\
\frac{1}{2}\int_t^T\int_{\dom}\phi_n^{\prime\prime}(y)Y^2\dx\d s&\to \frac{p(p-1)}{2}\int_t^T\int_{\dom}|y|^{p-2}|Y|^2\dx\d{s} \quad \text{ a.s.,}  \\
\int_t^T\int_{\dom}\phi_n^{\prime}(y)\alpha y \dx\d s&\to p \int_t^T\int_{\dom}\alpha |y|^{p}  \dx\d{s} \quad \text{ a.s.,} \\
\int_t^T\int_{\dom}\phi_n^{\prime}(y)\beta Y \dx\d s&\to p \int_t^T\int_{\dom}|y|^{p-2}\beta y Y\dx\d{s} \quad \text{ a.s.,}  
\end{align*}
for all $t\in[0,T]$. 

Finally, for the term $\int_t^T\int_{\dom}\phi_n^{\prime}(y)Y \dx \d W(s)$, first we will prove that
\begin{align}\label{bfr:dct}
\int_{t}^{T}\left|\langle \phi_n^{\prime}(y),Y\rangle_{L^2(\dom)}-\left\langle p|y|^{p-2}y,Y\right\rangle_{L^2(\dom)}\right|^2\d s\to 0 \quad \text{ a.s. for all $t\in[0,T]$}.
\end{align}
For this, we recall from \eqref{eq:p_bounds} that $|\phi^{\prime}_n(r)|\leq N|r|^{p-1}$ and by Cauchy-Schwarz inequality
\begin{align*}
\left|\langle \phi_n^{\prime}(y),Y\rangle_{L^2(\dom)}-\langle p|y|^{p-2}y,Y\rangle_{L^2(\dom)}\right|^2 & \leq C \left(\int_{\dom}|y|^{p-2}|Y|^2\dx\right) \left(\int_{\dom}|y|^{p}\dx\right)  \quad \textnormal{a.s.}
\end{align*}
On the other hand, by estimate \eqref{avoc}, we have
\begin{align*}
\int_{t}^{T}& \left(\int_{\dom}|y|^{p-2}|Y|^2\dx\right) \left(\int_{\dom}|y|^{p}\dx\right) \d{s} \\
&\leq C\left(\sup_{s\in[t,T]}\int_{\dom}|y(s)|^{p}\dx\right)\left(\int_t^T\int_{\dom}|y|^{p-2}|Y|^2\dx\d{s}\right), \quad \\
&<+\infty, \quad \text{ a.s. for all $t\in[0,T]$.}
\end{align*}
From \eqref{prop:convr}, the continuity of the $L^2$-inner product, and dominated convergence theorem allow us to conclude \eqref{bfr:dct} and therefore 
\begin{align*}
\int_t^T\int_{\dom}\phi_n^{\prime}(y)Y\d W(s)\to p\int_t^T\int_{\dom}|y|^{p-2}yY\d W(s)\;\;\text{ a.s. for all $t\in[0,T]$.}
\end{align*}
The proof is complete. 
\end{proof}

\subsection{Extra regularity for the process $Y$ and proof of \Cref{thm:ito}}
The goal of this part is to establish extra regularity for the process $Y$.

\begin{proposition}\label{prop:extra_regularity_Y}
Let $(y,Y)$ be a weak solution to \eqref{eq:intro} and assume that $y_T\in L^p_{\mathcal F_T}(\Omega;L^p(\dom))$ and $F\in L^{p}_{\fil}(0,T;L^p(\mathcal O))$ for some $p\in(2,+\infty)$. Then $Y\in L^p_{\fil}(\Omega;L^2(0,T;L^2(\dom)))$ and there is a constant $C>0$ only depending on $\dom$, $\alpha$, $\beta$, and $p$ such that
\begin{equation}\label{est:final_impr_Y}
\E\left[\left(\int_0^{T}\|Y(t)\|_{L^2(\dom)}^2\dt\right)^{p/2}\right]\leq Ce^{CT}\E\left(\|y_T\|^p_{L^p(\dom)}+\int_{0}^T\|F(t)\|_{L^p(\dom)}^p\dt\right).
\end{equation}
\end{proposition}

\begin{proof}
Let $(y,Y)$ be a weak solution to \eqref{eq:intro}. Along the proof, $C$ stands for a generic positive constant depending at most on $\mathcal O$, $\|\alpha\|_{L^\infty_{\fil}(0,T;L^\infty(\dom))}$, $\|\beta\|_{L^\infty_{\fil}(0,T;L^\infty(\dom))}$, and $p$ but which is uniform with respect to $T$. This constant may change from line to line.

Using formula \eqref{itosfor} in \Cref{prop:ito} with $p=2$ (which amounts to the usual Ito's formula) and dropping some of the positive terms we have
\begin{equation*}
\int_{t}^{T}\int_{\dom}|Y|^2\dx\d{s}\leq \int_{\dom}|y_T|^2\dx+2\int_{t}^{T}\int_{\dom} y(\alpha y+\beta Y+F) \dx\d{s}-2\int_{t}^{T}\int_{\dom} y Y \dx \d{W}(s), \quad \textnormal{a.s.}
\end{equation*}
for any $t\in[0,T]$. Since $y\in L^2_{\fil}(\Omega; C([0,T];L^2(\dom)))$, then $\int_{\dom}|y_T|^2\leq \sup_{s\in[0,T]}\|y(s)\|^2_{L^2(\dom)}$ a.s. From this, and taking the power $p/2$ with $p>2$ in both sides of the above inequality, we deduce
\begin{align}\notag 
&\left(\int_{t}^{T}\int_{\dom}|Y|^2\dx\d{s}\right)^{p/2}  \\ \notag
&\; \leq C \left( \sup_{s\in [0,T]}\|y(s)\|_{L^2(\dom)}^{p}+\left(\int_{0}^{T}\int_{\dom} |\alpha||y|^2\dx\d{s}\right)^{p/2}+\left(\int_{0}^{T}\int_{\dom}|\beta y Y|\dx\d{s}\right)^{p/2}\right)
\\ \label{ito_ineq}
&\quad +C\left(\left(\int_0^{T}\int_{\dom}|yF|\dx\d{s}\right)^{p/2}+\left|\int_{t}^{T}\int_{\dom} y Y \dx \d{W}(s)\right|^{p/2}\right),
\end{align}
almost surely for any $t\in[0,T]$.

Let us estimate the expectation of the last four terms in the right-hand side of \eqref{ito_ineq}. For the first one, from H\"older inequality and since $\alpha$ verifies \ref{Hy}, we readily deduce
\begin{equation}\label{est:nt1}
\E\left[\left(\int_{0}^{T}\int_{\dom} |\alpha||y|^2\dx\d{s}\right)^{p/2}\right]\leq C T^{p/2} \E\left(\sup_{s\in[0,T]}\|y(s)\|_{L^2(\dom)}^{p}\right).
\end{equation}
%

For the second one, recalling \ref{Hy} and using Cauchy-Schwarz inequality successively, we have
\begin{align}\notag
\left(\int_{0}^{T}\int_{\dom}|\beta y Y|\dx\d{s}\right)^{p/2}&\leq C \left(\int_0^{T}\|y(s)\|^2_{L^2(\dom)}\d{s}\right)^{p/4}\left(\int_0^{T}\|Y(s)\|^2_{L^2(\dom)}\d{s}\right)^{p/4}  \quad\textnormal{a.s.}
\end{align}
%
Therefore, taking expectation and using Cauchy-Schwarz and Young inequalities we obtain
\begin{align}\notag
&\E\left[\left(\int_{0}^{T}\int_{\dom}|\beta y Y|\dx\d{s}\right)^{p/2}\right]\\ \notag
&\qquad \leq \delta \E\left[\left(\int_0^{T}\|Y(s)\|^2_{L^2(\dom)}\d{s}\right)^{p/2}\right]+\frac{C}{\delta} \E\left[\left(\int_0^{T}\|y(s)\|^2_{L^2(\dom)}\d{s}\right)^{p/2}\right]
\end{align}
for any $\delta>0$.
 Using H\"older inequality in the last term, we can further estimate 
\begin{align}\notag
&\E\left[\left(\int_{0}^{T}\int_{\dom}|\beta y Y|\dx\d{s}\right)^{p/2}\right]\\ \label{est:nt2}
&\qquad \leq \delta \E\left[\left(\int_0^{T}\|Y(s)\|^2_{L^2(\dom)}\d{s}\right)^{p/2}\right]+\frac{CT^{p/2}}{\delta} \E\left(\sup_{s\in[0,T]}\|y(s)\|^p_{L^2(\dom)}\right).
\end{align}

To estimate the third term in the right-hand side of \eqref{ito_ineq}, we proceed as follows. Using Cauchy-Schwarz and H\"older inequalities, we get
\begin{align*}
\left(\int_{0}^{T}\int_{\dom}|yF|\dx\d{s}\right)^{p/2}\leq \sup_{s\in[0,T]}\|y(s)\|^{p/2}_{L^2(\dom)}\left(\int_0^{T}\|F(s)\|_{L^2(\dom)}\d{s}\right)^{p/2}, \quad \textnormal{a.s.}
\end{align*}
Taking expectation and using Cauchy-Schwarz and Young inequalities we obtain
\begin{align}\label{est:t1}
\E\left[\left(\int_{0}^{T}\int_{\dom}|yF|\dx\d{s}\right)^{p/2} \right] \leq \frac{1}{2}\E\left(\sup_{s\in[0,T]}\|y(s)\|^{p}_{L^2(\dom)}\right)+\frac{1}{2}\E\left[\left(\int_0^{T}\|F(s)\|_{L^2(\dom)}\d{s}\right)^{p}\right].
\end{align}

To estimate the last term in \eqref{ito_ineq}, arguing as we did in \eqref{lylys_trick}, we have that 
\begin{equation*}
\sup_{t\in[0,T]} \left|\int_{t}^{T}\int_{\dom} y Y \dx \d{W}(s)\right|^{p/2}\leq 2 \sup_{t\in[0,T]} \left|\int_{0}^{t}\int_{\dom} y Y \dx \d{W}(s)\right|^{p/2}, 
\end{equation*}
whence, by Burkholder-Davis-Gundy inequality, we deduce that
\begin{equation}\label{eq:est_abdg}
\E \left[\left|\int_{t}^{T}\int_{\dom} y Y \dx \d{W}(s)\right|^{p/2}\right] \leq C \E\left[\left|\int_{0}^{T}\left(\int_{\dom}yY\dx\right)^2\d{s}\right|^{p/4}\right].
\end{equation}
%
Using Cauchy-Schwarz and H\"older inequalities we can  estimate the right-hand side of \eqref{eq:est_abdg} as
\begin{align}\notag
\E \left[\left|\int_{t}^{T}\int_{\dom} y Y \dx \d{W}(s)\right|^{p/2}\right] 
& \leq C\E\left[\sup_{s\in[0,T]}\|y(s)\|_{L^2(\dom)}^{p/2}\left|\int_{0}^{T}\|Y(s)\|^2_{L^2(\dom)}\d{s}\right|^{p/4}\right],
\end{align}
and from Young inequality, we deduce that for any $\delta>0$
\begin{align}\label{est:t2}
\E \left[\left|\int_{t}^{T}\int_{\dom} y Y \dx \d{W}(s)\right|^{p/2}\right] 
& \leq \delta \E\left[\left|\int_{0}^{T}\|Y(s)\|^2_{L^2(\dom)}\d{s}\right|^{p/2}\right]+\frac{C}{\delta}\E\left(\sup_{s\in[0,T]}\|y(s)\|_{L^2(\dom)}^{p}\right).
\end{align}

Taking expectation in \eqref{ito_ineq} and putting together with \eqref{est:nt1}, \eqref{est:nt2}, \eqref{est:t1}, and \eqref{est:t2} we have that for any $t\in[0,T]$ the following holds
\begin{align*}\notag 
&\E\left[\left(\int_{t}^{T}\int_{\dom}|Y|^2\dx\d{s}\right)^{p/2}\right]  \\ \notag
&\quad \leq C\left(1+\frac{2T^{p/2}}{\delta}\right) \E \left( \sup_{s\in [0,T]}\|y(s)\|_{L^2(\dom)}^{p}\right)+2\delta \E\left[\left(\int_{0}^{T}\|Y(s)\|^2_{L^2(\dom)}\d{s}\right)^{p/2}\right]\\
&\qquad +\frac{1}{2}\E\left[\left(\int_0^{T}\|F(s)\|_{L^2(\dom)}\d{s}\right)^{p}\right].
\end{align*}
In particular, taking $t=0$ in the above estimate and setting $\delta>0$ small enough yields
\begin{align}\notag 
&\E\left[\left(\int_{0}^{T}\|Y(s)\|^2_{L^2(\dom)}\d{s}\right)^{p/2}\right] \\ \label{eq:est_minimum}
&\quad \leq C(1+T^{p/2}) \E \left( \sup_{s\in [0,T]}\|y(s)\|_{L^2(\dom)}^{p}\right)+ C\E\left[\left(\int_0^{T}\|F(s)\|_{L^2(\dom)}\d{s}\right)^{p}\right].
\end{align}
%

To conclude, using H\"older inequality and since $\|z\|^p_{L^2(\dom)} \leq |\dom|^{\frac{p-2}{2}}\|z\|^p_{L^p(\dom)}$ for any $z\in L^p(\dom)$, we obtain
\begin{align}\notag 
&\E\left[\left(\int_{0}^{T}\|Y(s)\|_{L^2(\dom)}^2\d{s}\right)^{p/2}\right] \\ \notag
&\quad  \leq C(1+T^{p/2}) \E \left( \sup_{s\in [0,T]}\|y(s)\|_{L^2(\dom)}^{p}\right)+ CT^{p-1}\E\left(\int_0^{T}\|F(s)\|_{L^2(\dom)}^p\d{s}\right)\\
\label{est:improved_Y} 
& \quad  \leq C^\prime e^{C^\prime T} \E \left(  \sup_{s\in [0,T]}\|y(s)\|_{L^p(\dom)}^{p} + \int_0^{T}\|F(s)\|_{L^p(\dom)}^p\d{s}\right),
\end{align}
for some constant $C^\prime>0$ only depending on $\mathcal O$, $\alpha$, $\beta$, and $p$. To conclude, we use estimate \eqref{avoc} in \eqref{est:improved_Y} to obtain the desired inequality \eqref{est:final_impr_Y}. This ends the proof.
\end{proof}

\begin{proof}[Proof of \Cref{thm:ito}]
The result follows directly from \Cref{prop:ito} and \Cref{prop:extra_regularity_Y}.
\end{proof}

\section{Uniform $L^p$- and $L^\infty$-estimates }\label{sec:estimates}
The goal of this section is to show \Cref{thm:infinity}. In the first step towards \Cref{thm:infinity}, we will use It\^{o}'s formula for the $L^p$-norm (see \Cref{prop:ito}) to obtain a precise uniform estimate for the solution to \eqref{eq:intro}. The result reads as follows. 

\begin{proposition}\label{prop:linf_uniform}
Let $(y,Y)$ be a weak solution to \eqref{eq:intro}. Assume that $y_T\in L^p_{\mathcal F_T}(\Omega;L^p(\dom))$ and $F\in L^{p}_{\fil}(0,T;L^p(\mathcal O))$. Then the process $y$ taken from the solution $(y,Y)$ to \eqref{eq:intro} satisfies
\begin{equation*}
\norme{y}_{L^\infty_{\fil}(0,T;L^p(\Omega;L^p(\mathcal O)))} \leq \exp (C T ) \left( \norme{y_T}_{L^{p}_{\mathcal F_T}(\Omega;L^p(\mathcal O))} + \norme{F}_{L^{p}_{\fil}(0,T;L^p(\mathcal O))}\right),
\end{equation*}
for a constant $C>0$ independent of $y_T$, $F$ and $p$. 
\end{proposition}

\begin{proof} Let $(y,Y)$ be a weak solution to \eqref{eq:intro}. We start from formula  \eqref{itosfor} in \Cref{prop:ito}, namely
\begin{align*}
&\int_{\dom}|y(t)|^p\dx+p(p-1)\int_t^T\int_{\dom}|y|^{p-2}|\nabla y|^2\dx\d s+\frac{p(p-1)}{2}\int_t^T\int_{\dom}|y|^{p-2}|Y|^2\dx\d s\\\notag
&=\int_{\dom}|y_T|^p\dx+p\int_t^T\int_{\dom}|y|^{p-2}y\left(\alpha y+\beta Y+F\right)\dx-p\int_t^T\int_{\dom}|y|^{p-2}yYdW(s), \quad \textnormal{a.s.,}
\end{align*}
that holds for any $t\in[0,T]$.

In particular, by taking into account that the third term in the left-hand side is nonnegative, we deduce that
\begin{align}\notag 
&\int_{\dom}|y(t)|^p\dx +\frac{p(p-1)}{2}\int_t^T\int_{\dom}|y|^{p-2}|Y|^2\dx\d s \\ \label{eq:itosfortineq}
&\quad \leq\int_{\dom}|y_T|^p\dx+p\int_t^T\int_{\dom}|y|^{p-1}\left|\alpha y+\beta Y+F\right|\dx-p\int_t^T\int_{\dom}|y|^{p-2}yY\d W(s), \quad \text{a.s.}
\end{align}
Let $r \leq t$ and take conditional expectation with respect to $\mathcal F_r$ in \eqref{eq:itosfortineq} to get
\begin{align}\notag
&\mathbb E \left( \norme{y(t)}_{L^p(\mathcal{O})}^p  \bigmid \mathcal F_r \right) +\frac{p(p-1)}{2}\E \left(\int_t^T\int_{\dom}|y|^{p-2}|Y|^2\dx\d s \bigmid \mathcal F_r \right) \\ 
&\quad \leq \mathbb E \left( \norme{y_T}_{L^p(\mathcal{O})}^p \bigmid \mathcal F_r \right) + p\, \mathbb E \left( \int_t^{T} \int_{\mathcal O}  |y|^{p-1} \big[|\alpha y|+|\beta Y|+|F|\big] \dx \d{s}  \bigmid \mathcal F_r  \right),\label{eq:itosfor2}
\end{align}
because
\begin{equation*}
\mathbb E \left( \int_t^{T} \int_{\mathcal O} p |y|^{p-2} y Y dW(s) \bigmid \mathcal F_r \right) = 0.
\end{equation*}
This can be shown by employing estimate \eqref{avoc} and following the steps in \Cref{pf:mart}. For brevity, we skip the details. 
 
Let us estimate the second integral in \eqref{eq:itosfor2}. By a direct computation using the monotonicty of the conditional expectation, we have
\begin{equation}\label{est:alphay}
p\, \mathbb E \left( \int_t^{T} \int_{\mathcal O}  |y|^{p-1} |\alpha y| \dx \d{s}  \bigmid \mathcal F_r  \right) \leq p \, \|\alpha\|_{L^\infty_{\fil}(0,T;L^\infty(\dom))} \, \mathbb E \left( \int_t^{T} \int_{\mathcal O}  |y|^{p} \dx \d{s}  \bigmid \mathcal F_r  \right).
\end{equation}
For the second one, by Cauchy-Schwarz and Young inequality we have 
\begin{align}\notag
&p\, \mathbb E \left( \int_t^{T} \int_{\mathcal O}  |y|^{p-1} |\beta Y|  \dx \d{s}  \bigmid \mathcal F_r  \right)\\ \label{est:betaY}
&\quad \leq \frac{p}{2}\,\E \left( \int_t^{T} \int_{\mathcal O}  |y|^{p-2} |Y|^2  \dx \d{s}  \bigmid \mathcal F_r  \right) + \frac{p \|\beta\|^2_{L^\infty_{\fil}(0,T;L^\infty(\dom))}}{2}  \E \left( \int_t^{T} \int_{\mathcal O}  |y|^{p}  \dx \d{s}  \bigmid \mathcal F_r  \right).
\end{align}
For the last one, using Young inequality $a b \leq (1/p) a^p + (1/q) b^q$ with $q=p/(p-1)$, so $p a b \leq a^p + (p-1) b^{p/(p-1)}$ with $a=|F|$ and $b = |y|^{p-1}$ we obtain  that
\begin{align}\notag 
&p\, \mathbb E \left( \int_t^{T} \int_{\mathcal O}  |y|^{p-1} |F|  \dx \d{s}  \bigmid \mathcal F_r  \right) \\ \label{eq:est_F}
\quad &\leq  \E \left( \int_t^{T} \int_{\mathcal O} |F|^p \dx \d{s}  \bigmid \mathcal F_r  \right) + (p-1)  \E \left( \int_t^{T} \int_{\mathcal O} 
  |y|^{p}  \dx \d{s}  | \mathcal F_r  \right).
\end{align}
Therefore, combining \eqref{eq:itosfor2} with estimates \eqref{est:alphay}--\eqref{eq:est_F} we get
\begin{align*}\notag
&\mathbb E \left( \norme{y(t)}_{L^p(\mathcal{O})}^p  \bigmid \mathcal F_r \right) +\frac{p(p-2)}{2}\E \left(\int_t^T\int_{\dom}|y|^{p-2}|Y|^2\dx\d s \bigmid \mathcal F_r \right) \\ \notag
&\quad \leq \mathbb E \left( \norme{y_T}_{L^p(\mathcal{O})}^p \bigmid \mathcal F_r \right) + \E \left( \int_t^{T} \int_{\mathcal O} |F|^p \dx \d{s}  \bigmid \mathcal F_r  \right)  \\
&\qquad +  [pK+(p-1)]\, \mathbb E \left( \int_t^{T} \int_{\mathcal O}  |y|^{p} \dx\d{s}  \bigmid \mathcal F_r  \right),
\end{align*}
where $K:=\|\alpha\|_{L^\infty_{\fil}(0,T,L^\infty(\dom))}+\|\beta\|^2_{L^\infty_{\fil}(0,T,L^\infty(\dom))}$. Note that since $p\geq 2$, the second term in the left-hand side of the above inequality is always nonnegative. 

Now, we take $r=t$ in the previous estimate and use the fact that $\mathbb E ( \norme{y(t)}_{L^p(\mathcal{O})}^p  \mid \mathcal F_t ) = \norme{y(t)}_{L^p(\mathcal{O})}^p$ 
since $y$ is $\F_t$-adapted to obtain 
\begin{align}\notag 
&\norme{y(t)}_{L^p(\mathcal{O})}^p  \\
&\leq  \E \left( \norme{y_T}_{L^p(\mathcal{O})}^p \bigmid \mathcal F_t \right) +  \E \left( \int_t^{T} \int_{\mathcal O} |F|^p \dx \d{s}  \bigmid \mathcal F_t  \right) +[pK+(p-1)]  \E \left( \int_t^{T} \int_{\mathcal O} 
  |y|^{p}  \dx \d{s}  \bigmid \mathcal F_t  \right) \label{est:lp_ft}
  \end{align}
whence, by taking expectation and using Fubini's theorem yield, we get
\begin{equation*}
\mathbb E \left( \norme{y(t)}_{L^p(\mathcal{O})}^p\right) \leq  \norme{y_T}_{L^{p}_{\mathcal F_T}(\Omega;L^p(\mathcal O))}^p +  \norme{F}_{L^{p}_{\fil}(0,T;L^p(\mathcal O))}^p + [pK+(p-1)] \int_t^{T} \mathbb E\left( \norme{y(s)}_{L^p(\mathcal{O})}^p\right) \d{s},
\end{equation*}
  for all $t\in[0,T]$.
  
To finish, we apply backward Grönwall estimate (see \Cref{lem:back_gron}) to deduce
\begin{equation*}
\E \left( \norme{y(t)}_{L^p(\mathcal{O})}^p\right) \leq \exp([pK+(p-1)](T-t)) \left( \norme{y_T}_{L^{p}_{\F_T}(\Omega;L^p(\mathcal O))}^p + \norme{F}_{L^{p}_{\fil}(0,T;L^p(\mathcal O))}^p  \right),
\end{equation*}
that is 
\begin{equation*}
\norme{y(t)}_{L^p(\Omega; L^p(\mathcal{O}))}^p \leq \exp([pK+(p-1)](T-t)) \left( \norme{y_T}_{L^{p}_{\F_T}(\Omega;L^p(\mathcal O))}^p + \norme{F}_{L^{p}_{\fil}(0,T;L^p(\mathcal O))}^p  \right).
\end{equation*}
We take the power $1/p$ and we get
\begin{align*}
\norme{y(t)}_{L^p(\Omega; L^p(\mathcal{O}))} &\leq \exp((K+(p-1)/p) (T-t)) \left( \norme{y_T}_{L^{p}_{\F_T}(\Omega;L^p(\mathcal O))}^p + \norme{F}_{L^{p}_{\fil}(0,T;L^p(\mathcal O))}^p  \right)^{1/p} \\
& \leq \exp([K+1](T-t)) \left( \norme{y_T}_{L^{p}_{\F_T}(\Omega;L^p(\mathcal O))} + \norme{F}_{L^{p}_{\fil}(0,T;L^p(\mathcal O))}  \right),
\end{align*}
where we have used that $(p-1)/p<1$. Taking $L^\infty$-norm in time yields the desired result.
\end{proof}

We present the proof of our main result. 

\begin{proof}[Proof of \Cref{thm:infinity}]

From \Cref{prop:linf_uniform}, we have that the following estimate holds
\begin{equation}\label{eq:limit_initial}
\norme{y(t)}_{L^p(\Omega;L^p(\mathcal O))} \leq \exp (C T ) \left( \norme{y_T}_{L^{p}_{\mathcal F_T}(\Omega;L^p(\mathcal O))} + \norme{F}_{L^{p}_{\fil}(0,T;L^p(\mathcal O))}\right)
\end{equation}
for all $t\in[0,T]$, where $y$ can be found from $(y,Y)$ solution to \eqref{eq:intro} and $C>0$ is a constant independent of $p$. 

By H\"older inequality, the right-hand side of the above expression is bounded as
\begin{align}\notag
R.H.S. &\leq \exp (C T )  \left(|\mathcal O|^{1/p} \norme{y_T}_{L^{\infty}_{\mathcal F_T}(\Omega;L^\infty(\mathcal O))} + |\mathcal O|^{1/p} |T|^{1/p} \norme{F}_{L^{\infty}_{\fil}(0,T;L^\infty(\mathcal O))}\right)\\ \label{eq:uniform_est}
&\leq   \exp (C^\prime(1+T)) \left( \norme{y_T}_{L^{\infty}_{\mathcal F_T}(\Omega;L^\infty(\mathcal O))} + \norme{F}_{L^{\infty}_{\fil}(0,T;L^\infty(\mathcal O))} \right)
\end{align}
for a constant $C^\prime>0$ depending on $\dom$ but uniform with respect to $T$ and $p$. Estimate \eqref{eq:uniform_est} and \Cref{lem:limit} allows us to pass to the limit as $p\to+\infty$ in \eqref{eq:limit_initial} and obtain 
\begin{equation*}
\norme{y(t)}_{L^\infty(\Omega;L^\infty(\mathcal O))} \leq \exp (C T ) \left( \norme{y_T}_{L^{\infty}_{\mathcal F_T}(\Omega;L^{\infty}(\mathcal O))} + \norme{F}_{L^{\infty}_{\fil}(0,T;L^{\infty}(\mathcal O))}\right)
\end{equation*}
for every $t\in[0,T]$. Therefore, 
\begin{equation*}
\norme{y}_{L^\infty_{\fil}(0,T;L^\infty(\Omega;L^\infty(\mathcal O)))} \leq \exp (C T ) \left( \norme{y_T}_{L^{\infty}_{\mathcal F_T}(\Omega;L^{\infty}(\mathcal O))} + \norme{F}_{L^{\infty}_{\fil}(0,T;L^{\infty}(\mathcal O))}\right).
\end{equation*}
The result \eqref{eq:linf_est} then follows by interchanging the $L^{\infty}$-norms.
\end{proof}

\section{Local well-posedness of semilinear equations}\label{sec:semilinear}
In this section, we present the proof of \Cref{thm:semi}. Hereinafter, we set $Q_T:=(0,T)\times \dom$ and $\Sigma_T:=(0,T)\times \partial \dom$ to abridge the notation.

We begin by stating the following result.

\begin{proposition}\label{thm:new_space}
Let $p\in[2,+\infty)$, $(y,Y)$ be a weak solution to \eqref{eq:intro} and assume that $y_T\in L^\infty_{\mathcal F_T}(\Omega;L^\infty(\dom))$ and $F\in L^{\infty}_{\fil}(0,T;L^\infty(\mathcal O))$. Then, $(y,Y)\in\mathscr Y_p$ and satisfies
\begin{equation}\label{est:Yp_full}
\|(y,Y)\|_{\mathscr Y_p}\leq  C\left(\|y_T\|_{L^\infty_{\F_T}(\Omega;L^\infty(\dom))}+\|F\|_{L^\infty_{\fil}(0,T;L^\infty(\dom))}\right),
\end{equation}
for a constant $C>0$ only depending on $\mathcal O$, $T$, and $p$.
\end{proposition}

\begin{proof}
The proof is a direct consequence of H\"older inequality, \Cref{thm:l2_sol}, point (ii) of \Cref{thm:ito}, and \Cref{prop:extra_regularity_Y}.
\end{proof}


\begin{proof}[Proof of \Cref{thm:semi}]
We will see that if the terminal data $y_T$ is small enough, then there is a unique solution $(y,Y)$ to \eqref{eq:semi} in $\mathscr Y_p$. We split the proof in two parts.

-\textit{Existence.} By Taylor's formula at second order, we have that
\begin{equation}\label{eq:taylor}
\forall s\in\mathbb R, \quad f(s)=f^\prime(0)s+s^2G(s), \quad 
\end{equation}
where $G(s)=\int_0^1 (1-\sigma) f^{\prime\prime}(\sigma s)\d{\sigma}$. For any $r\in(0,1)$, let 
\begin{equation}\label{Y_r}
\mathscr Y^r_p:=\{(y,Y)\in\mathscr Y_p: \|(y,Y)\|_{\mathscr Y_p}\leq r\},
\end{equation} and consider the linear equation
\begin{equation}
\label{eq:linearized}
\begin{cases}
\d{y}=-\left(\Delta y + \alpha y + \beta Y  + f^\prime(0) y + \ov{y}^2G(\ov y)  \right)\d{t}+Y\d{W}(t) &\text{in } Q_T, \\
y=0 &\text{on } \Sigma_T, \\
y(T)=y_T &\text{in }\dom,
\end{cases}
\end{equation}
where $\ov y$ is the first entry of any given pair $(\ov y, \ov Y)\in\mathscr Y_p^r$ and $y_T\in L^\infty_{\F_T}(\Omega;L^\infty(\dom))$. Note that this is a linearized version of \eqref{eq:semi}. 

Let us denote by $(y,Y)$ the solution to \eqref{eq:linearized}. We will check that the mapping
\begin{equation*}
\begin{split}
\mathcal N : \mathscr Y_p^r &\longrightarrow \mathscr Y_p^r \\
 (\ov y,\ov Y)&\longmapsto (y,Y)
\end{split}
\end{equation*}
is well-defined if we consider terminal data $y_T$ small enough. Since $f\in C^\infty(\mathbb R)$, from point \ref{eq:est_sec} of \Cref{lem:nonlinearity}, we have that $|G(s)|\leq M$ for all $s\in[-2,2]$ where the constant $M>0$ only depends on $f$. Hence, defining $\ov F:=\ov y^2G(\ov y)$ we have
\begin{align}\notag 
\|\ov{F}\|_{L^\infty_{\fil}(0,T;L^\infty(\dom))}&=\|\ov y^2G(\ov y)\|_{L^\infty_{\fil}(0,T;L^\infty(\dom))} \\ \notag
&\leq M \|\ov{y}^2\|_{L^\infty_{\fil}(0,T;L^\infty(\dom))} =M \|\ov{y}\|^2_{L^\infty_{\fil}(0,T;L^\infty(\dom))} \\ \label{eq:est_infty}
&\leq M\|(\ov y,\ov Y)\|_{\mathscr Y_p}^2\leq M r^2 < +\infty,
\end{align}
for any $(\ov y, \ov Y)\in\mathscr Y_p^r$. 

Thus, using \Cref{thm:new_space} with this particular $\ov{F}$ and estimate \eqref{eq:est_infty} yield that there is a solution $(y,Y)\in\mathscr Y_p$ to \eqref{eq:linearized} such that 
\begin{align}\label{eq:est_y}
\|(y,Y)\|_{\mathscr Y_p}&\leq C\left(\|y_T\|_{L^\infty_{\F_T}(\Omega;L^\infty(\dom))}+Mr^2\right)
\end{align}
for some $C>0$ only depending on $\dom$, $T$ and $p$. Let us take 
\begin{equation*}
r\leq \frac{1}{8CM},
\end{equation*}
 where $C$ is the constant appearing in \eqref{eq:est_y} and fix $\delta>0$ such that $\delta \leq r/2C$. Thus for any $y_T\in L^\infty_{\F_T}(\Omega;L^\infty(\dom))$ verifying $\|y_T\|_{L^\infty_{\F_T}(\Omega;L^\infty(\dom))}\leq \delta$ we have
\begin{equation}\label{eq:est_y_r}
\|(y,Y)\|_{\mathscr Y_p} \leq \frac{r}{2}+\frac{r}{8}=  \frac{5r}{8}\leq r.
\end{equation}
This shows that the map $\mathcal N:\mathscr Y_p^r\to \mathscr Y_p^r$ is well-defined. 

To conclude, we will verify that $\mathcal N$ is a contraction on $\mathscr Y_p^r$ and use Banach fixed point theorem. By point \ref{eq:est_dif_G} of \Cref{lem:nonlinearity} with $\mathcal I=[-2,2]$, we have that for  any $\ov y_1$, $\ov y_2$ taken from the pairs $(\ov{y}_{i},\ov{Y}_{i})\in\mathscr Y_p^r$, $i=1,2$,
\begin{align}\notag
&\|\ov{y}_1G(\ov y_1)-\ov{y}_2 G(\ov y_2)\|_{L^\infty_{\fil}(0,T;L^\infty(\dom))}\\ \notag
&\qquad \leq M \|\ov{y}_1-\ov{y}_2\|_{L^\infty_{\fil}(0,T;L^\infty(\dom))}\left(\|\ov{y}_1\|_{L^\infty_{\fil}(0,T;L^\infty(\dom))}+\|\ov{y}_2\|_{L^\infty_{\fil}(0,T;L^\infty(\dom))}\right) \\ \label{est:rhs}
&\quad \qquad + M_1\|y_2\|^2_{L^\infty_{\fil}(0,T;L^\infty(\dom))} \|\ov{y}_1-\ov{y}_2\|_{L^\infty_{\fil}(0,T;L^\infty(\dom))},
\end{align}
for some positive constants $M,M_1$ only depending on $f$. 

Denote by $(y_i,Y_i)$ the solution of \eqref{eq:linearized} for the corresponding $\ov{y}_i$, $i=1,2$. From the linearity of \eqref{eq:linearized}, \Cref{thm:new_space} with $\ov{F}=\ov{y}_1G(\ov y_1)-\ov{y}_2 G(\ov y_2)$ and $y_T\equiv 0$, and estimate \eqref{est:rhs}, we can compute
\begin{align}\notag
&\|\mathcal N(\ov{y}_1,\ov{Y}_1)-\mathcal N(\ov{y}_2,\ov{Y}_2)\|_{\mathscr Y_p}= \|\mathcal ({y}_1,{Y}_1)- ({y}_2,{Y}_2)\|_{\mathscr Y_p} \\ \notag
&\qquad \leq CM \|\ov{y}_1-\ov{y}_2\|_{L^\infty_{\fil}(0,T;L^\infty(\dom))}\left(\|\ov{y}_1\|_{L^\infty_{\fil}(0,T;L^\infty(\dom))}+\|\ov{y}_2\|_{L^\infty_{\fil}(0,T;L^\infty(\dom))}\right) \\ \notag
&\quad \qquad + CM_1\|y_2\|^2_{L^\infty_{\fil}(0,T;L^\infty(\dom))} \|\ov{y}_1-\ov{y}_2\|_{L^\infty_{\fil}(0,T;L^\infty(\dom))} \\\label{eq:est_r}
& \qquad \leq 2CM r \|\ov{y}_1-\ov{y}_2\|_{L^\infty_{\fil}(0,T;L^\infty(\dom))}+ CM_1 r^2 \|\ov{y}_1-\ov{y}_2\|_{L^\infty_{\fil}(0,T;L^\infty(\dom))}.
\end{align}
From \eqref{eq:est_r} and decreasing (if necessary) the value of $r$ and taking $r\leq \min\{\frac{1}{8CM},\frac{1}{\sqrt{4CM_1}}\}$, we get
\begin{align*}
\|\mathcal N(\ov{y}_1,\ov{Y}_1)-\mathcal N(\ov{y}_2,\ov{Y}_2)\|_{\mathscr Y_p} \leq \frac{1}{2}\|\ov{y}_1-\ov{y}_2\|_{L^\infty_{\fil}(0,T;L^\infty(\dom))} \leq  \frac{1}{2} \| (\ov y_1,\ov Y_1)-(\ov y_2,\ov Y_2)\|_{\mathscr Y_p}.
\end{align*}
Thus $\mathcal N$ is a contraction on the set $\mathscr Y_p^r$. A direct application of Banach fixed point theorem yields that there is a unique fixed point for $\mathcal N$ which is a weak solution to \eqref{eq:semi} thanks to \eqref{eq:taylor} and that belongs to $\mathscr Y_p^r$ since  \eqref{eq:est_y_r} holds. 

\smallskip

-\textit{Uniqueness}. Let $\delta>0$ be fixed and small enough as in the previous part. Let $y_T\in L^\infty_{\mathcal F_T}(\Omega;L^\infty(\dom))$ be such that $\|y_T\|_{L^\infty_{\F_T}(\Omega;L^\infty(\dom))}\leq\delta$ and assume that there are two weak solutions $(y_1,Y_1)$ and $(y_2,Y_2) $ to \eqref{eq:semi} for this same data that belong to some ball $\mathscr Y_p^r$ (recall \eqref{Y_r}). Defining $(y,Y):=(y_1-y_2,Y_1-Y_2)$, it is not difficult to see that these new variables verify 
\begin{equation*}
\begin{cases}
\d{y}=-\left(\Delta y + \alpha y + \beta Y + \left(f(y_1)-f(y_2)\right) \right)\d{t}+Y\d{W}(t) &\text{in } Q_T, \\
y=0 &\text{on } \Sigma_T, \\
y(T)=0 &\text{in }\dom.
\end{cases}
\end{equation*}
Arguing as we did for obtaining \eqref{est:lp_ft} and taking expectation, we can deduce
\begin{align}\notag 
&\mathbb E\left(\norme{y(t)}_{L^p(\mathcal{O})}^p\right) \leq \E \left( \int_t^{T} \int_{\mathcal O} |f(y_1)-f(y_2)|^p \dx \d{s} \right) +C  \E \left( \int_t^{T} \int_{\mathcal O} 
  |y|^{p}  \dx \d{s}\right),
  \end{align}
where $C>0$ is a constant only depending on $p$, $\alpha$ and $\beta$. Since $(y_i,Y_i)\in \mathscr Y_p^r$, $i=1,2$, and using the fact that $f$ is locally Lipschitz on $[-r,r]$, there is a constant $C^\prime>0$ depending on $r$ such that $|f(s_1)-f(s_2)|\leq C^\prime|s_1-s_2|$, $s_1,s_2\in\mathbb [-r,r]$, therefore
\begin{align}\notag 
&\mathbb E\left(\norme{y(t)}_{L^p(\mathcal{O})}^p\right) \leq C^{\prime\prime}\E \left( \int_t^{T} \int_{\mathcal O} |y|^p \dx \d{s} \right),
  \end{align}
  for a constant $C^{\prime\prime}>0$. A standard Grönwall argument using \Cref{lem:back_gron}, yields that $y_1=y_2$ a.e $(t,x)\in Q_T$, a.s. This ends the proof.
\end{proof}

\section{Applications to controllability}\label{sec:control}

{The goal of this section is to present some applications to the control of backward stochastic parabolic equations, first in the linear case and then in the nonlinear framework. In particular, we will prove \Cref{thm:control} and \Cref{theo:control_semi}.}

\subsection{The linear case}

We consider the following backward stochastic heat equation
 \begin{equation}
\label{eq:backward_linear}
\begin{cases}
\d{y}=-\left(\Delta y + \alpha y + \chi_{\dom_0}h\right)\d{t}+Y\d{W}(t) &\text{in } Q_T, \\
y=0 &\text{on } \Sigma_T, \\
y(T)=y_T &\text{in }\dom.
\end{cases}
\end{equation}
In \eqref{eq:backward_linear}, we recall that $(y,Y)$ the state variable while $h$ is the control variable located in $\mathcal O_0$ where $\mathcal O_0$ is a nonempty open subset of $\mathcal O$.

We now present the proof of the null-controllability result of \Cref{thm:control}.
\begin{proof}[Proof of \Cref{thm:control}]
We take $p \in (2, +\infty)$ so $p' \in (1,2)$.

The adjoint of \eqref{eq:backward_linear} is given by the random heat equation
\begin{equation}
\label{eq:forward_linear}
\begin{cases}
\d{q}=(\Delta q + \alpha q) \dt &\text{in } Q_T, \\
q=0 &\text{on }\Sigma_T, \\
q(0)=q_0 &\text{in }\dom.
\end{cases}
\end{equation}
Note that the randomness in equation \eqref{eq:forward_linear} appears only through the adapted coefficient $\alpha \in L_{\mathbb F}^{\infty}(0,T;L^{\infty}(\mathcal O))$ and the $\mathcal F_0$-measurable initial datum $q_0 \in L^{p'}_{\F_0}(\Omega;L^{p'}(\mathcal O))$. Moreover, for such coefficients, the weak solution
$t\mapsto q(t,\cdot)\in L^2(\mathcal O)$ to \eqref{eq:forward_linear}
is $\mathcal F_t$-measurable for every $t\in[0,T]$.

For $\omega \in \Omega$ fixed, we have from \cite[Proposition 3.2]{FCZ00}
\begin{equation}
\label{eq:observabilitydeterministic}
\norme{q(\omega,T)}_{L^{2}(\mathcal O)} \leq C \norme{q(\omega)}_{L^{1}((0,T)\times\mathcal O_0)},
\end{equation}
from which we directly deduce, by using $L^2(\mathcal O) \hookrightarrow L^{p'}(\mathcal O)$ and $L^{p'}((0,T)\times\mathcal O_0)) \hookrightarrow L^{1}((0,T)\times\mathcal O_0))$,
\begin{equation*}
\norme{q(\omega,T)}_{L^{p'}(\mathcal O)} \leq C \norme{q(\omega)}_{L^{p'}((0,T)\times\mathcal O_0)}.
\end{equation*}
By taking the power $p'$ and by taking the expectation we finally get 
\begin{equation*}
\mathbb E \Big[\norme{q(T)}_{L^{p'}(\mathcal O)}^{p'}\Big] \leq C \mathbb E \Big[\norme{q}_{L^{p'}((0,T)\times\mathcal O_0)}^{p'}\Big].
\end{equation*}
This rewrites as
\begin{equation}
\label{eq:obsLp'}
\norme{q(T)}_{L^{p'}_{\F_T}(\Omega;L^{p'}(\mathcal O))} \leq C  \norme{q}_{L^{p'}_{\fil}(0,T;L^p(\mathcal O_0))}.
\end{equation}
Note that the constant $C>0$ in \eqref{eq:obsLp'} is of the form $\exp(C/T)$ for $T \in (0,1)$ for some $C$ depending on $\mathcal O$, $\mathcal O_0$, $\alpha$ and $p$ because the observability constant $C$ in \eqref{eq:observabilitydeterministic} appearing in the deterministic observability inequality is of the form $\exp(C/T)$.

Now, we will follow an argument given by \cite{TZ09}, adapted to the $L^{p'}$-setting. For any $y_T \in L^{p}_{\F_T}(\Omega;L^{p}(\mathcal O))$, the goal is to find a control $h \in L^{p}_{\fil}(0,T;L^p(\mathcal O_0))$ such that the solution $y$ of \eqref{eq:backward_linear} satisfies $y(0) = 0$ a.s.

We introduce the following linear subspace of $L_{\fil}^{p'}(0,T;L^{p'}(\mathcal O_0))$, 
\begin{equation*}
\mathcal X = \{ q_{|\Omega \times (0,T) \times \mathcal O_0} \in L_{\fil}^{p'}(0,T;L^{p'}(\mathcal O))\ ;\ q\ \text{satisfies}\ \eqref{eq:forward_linear}\ \text{for some}\ q_0 \in L^{p'}_{\F_0}(\Omega;L^{p'}(\mathcal O))\}.
\end{equation*}
Then we define the following linear functional on $\mathcal X$ by
\begin{equation*}
\mathcal{L}(q) = - \mathbb E \int_{\mathcal O} y_T q(T) \dx.
\end{equation*}
The observability estimate \eqref{eq:obsLp'} leads to the fact that $\mathcal L$ is a bounded linear functional on $\mathcal X$ because
\begin{equation}
\label{eq:appobslinearfunctional}
 \left| \mathbb E  \int_{\mathcal O} y_T q(T) \dx \right| \leq \norme{y_T}_{L^{p}_{\F_T}(\Omega;L^{p}(\mathcal O))}\norme{q(T)}_{L^{p'}_{\F_T}(\Omega;L^{p'}(\mathcal O))} \leq C \norme{y_T}_{L^{p}_{\F_T}(\Omega;L^{p}(\mathcal O))} \norme{q}_{L^{p'}_{\fil}(0,T;L^{p'}(\mathcal O_0))}.
\end{equation}
By the Hahn-Banach theorem, $\mathcal L$ can be extended to a bounded linear functional on $L_{\fil}^{p'}(0,T;L^{p'}(\mathcal O_0))$ with
\begin{equation}
\label{eq:normfunctional}
\| \mathcal L\|_{\mathcal L( L_{\fil}^{p'}(0,T;L^{p'}(\mathcal O_0)),\R)}\leq C \norme{y_T}_{L^{p}_{\F_T}(\Omega;L^{p}(\mathcal O))},
\end{equation}
where $C$ is the same as in  \eqref{eq:obsLp'}. Now, we use the following fact coming from \cite[Corollary 2.3]{LYZ12}
\begin{equation*}
(L_{\fil}^{p'}(0,T;L^{p'}(\mathcal O_0))' = L_{\fil}^{p}(0,T;L^p(\mathcal O_0)).
\end{equation*}
This means that there exists $h \in L_{\fil}^{p}(0,T;L^p(\mathcal O_0))$ such that
\begin{equation}
\label{eq:equationcontrolFunctional}
\mathcal{L}(q) = \mathbb E \int_{0}^{T} \int_{\mathcal O_0} h q \dx \dt\qquad \forall q \in L^{p'}_{\fil}(0,T;L^{p'}(\mathcal O_0)),
\end{equation}
in particular for $q$ solution to \eqref{eq:forward_linear}
\begin{equation}
\label{eq:equationcontrol}
\mathcal{L}(q) = - \mathbb E \int_{\mathcal O} y_T q(T) \dx = \mathbb E \int_{0}^{T} \int_{\mathcal O_0} h q \dx \dt\qquad \forall q_0 \in L^{p'}_{\F_0}(\Omega;L^{p'}(\mathcal O)).
\end{equation}
Now we use It\^o's formula to get
\begin{equation*}
\d(yq) = y \d{q} + q \d{y} + \d{y} \d{q}.
\end{equation*}
So we get
\begin{equation*}
 \mathbb E \int_{\mathcal O} y_T q(T) \dx -  \mathbb E \int_{\mathcal O} y(0) q_0 \dx = \mathbb E \int_{Q_T}  y \d{q} + q \d{y} + \d{y} \d{q} = \mathbb E \int_{0}^{T} \int_{\mathcal O_0} h q \dx \dt.
\end{equation*}
Therefore, we have
\begin{equation*}
 \mathbb E \int_{\mathcal O} y(0) q_0 \dx = 0\qquad \forall q_0 \in L^{p'}_{\F_0}(\Omega;L^{p'}(\mathcal O)).
\end{equation*}
This implies that $y(0) = 0$ in $\mathcal O$ a.s. Moreover, we have from 
\eqref{eq:normfunctional} and \eqref{eq:equationcontrolFunctional} that
\begin{equation*}
\|h\|_{L_{\fil}^{p}(0,T;L^p(\mathcal O_0))} \leq C \norme{y_T}_{L^{p}_{\F_T}(\Omega;L^{p}(\mathcal O))},
\end{equation*}
where
\begin{equation*}
C=C(T) \leq \exp(C/T),
\end{equation*}
hence the expected bound on $h$ in \eqref{eq:EsimationControlhLp}. This concludes the proof in the case $p\in(2,+\infty)$.
\medskip

Now we deal with the more subtle case $p=+\infty$. The difficulty is that one cannot use  \cite[Corollary 2.3]{LYZ12} because the so-called Radon-Nikodym property is not satisfied by $L^{\infty}$, in particular $(L_{\fil}^{1}(0,T;L^{1}(\mathcal O_0))'$ is not equal to $L_{\fil}^{\infty}(0,T;L^\infty(\mathcal O_0))$. We start from the $L^1$-observability estimate for solution $q$ to \eqref{eq:forward_linear}
\begin{equation}
\label{eq:obsL1}
\norme{q(T)}_{L^{1}_{\fil}(\Omega;L^{1}(\mathcal O))} \leq C  \norme{q}_{L^{1}_{\fil}(0,T;L^1(\mathcal O_0))},
\end{equation}
that is an easy consequence of \eqref{eq:observabilitydeterministic}. We then introduce the following linear subspace of $L_{\mathcal F}^1(0,T;L^{1}(\mathcal O_0))$, 
\begin{equation*}
\mathcal X = \{ q_{|\Omega \times (0,T) \times \mathcal O_0} \in L_{\fil}^1(0,T;L^{1}(\mathcal O))\ :\ q\ \text{satisfies}\ \eqref{eq:forward_linear}\ \text{for some}\ q_0 \in L^1_{\F_0}(\Omega;L^{1}(\mathcal O))\}.
\end{equation*}
Then we define the following linear functional on $\mathcal X$ by
\begin{equation*}
\mathcal{L}(q) = - \mathbb E \int_{\mathcal O} y_T q(T) dx.
\end{equation*}
The observability estimate \eqref{eq:obsL1} leads to the fact that $\mathcal L$ is a bounded linear functional on $\mathcal X$. By the Hahn-Banach theorem, $\mathcal L$ can be extended to a bounded linear functional on $L_{\mathcal F}^1(0,T;L^{1}(\mathcal O_0))$ with
\begin{equation}
\label{eq:normfunctionalL1}
\| \mathcal L\|_{\mathcal L( L_{\fil}^{1}(0,T;L^{1}(\mathcal O_0)),\R)} \leq C \norme{y_T}_{L^{\infty}_{\F_T}(\Omega;L^{\infty}(\mathcal O))},
\end{equation}
where $C$ is the same as in \eqref{eq:obsL1}. In particular, $\mathcal L$ is a bounded linear functional on $L_{\fil}^2(0,T;L^{2}(\mathcal O_0))$ because $L_{\fil}^2(0,T;L^{2}(\mathcal O_0)) \hookrightarrow L_{\fil}^1(0,T;L^{1}(\mathcal O_0))$, therefore we can use the following fact coming from \cite[Corollary 2.3]{LYZ12},
\begin{equation*}
(L_{\fil}^2(\Omega;L^{2}(0,T;L^{2}(\mathcal O_0)))' = L_{\fil}^{2}(\Omega;L^2(0,T;L^2(\mathcal O_0))).
\end{equation*}
This means that there exists $h \in L_{\fil}^{2}(\Omega;L^2(0,T;L^2(\mathcal O_0)))$ such that
\begin{equation}
\label{eq:equationcontrolFunctionalL2}
\mathcal{L}(q) = \mathbb E \int_{0}^{T} \int_{\mathcal O_0} h q dx dt\qquad \forall q \in L^{2}_{\fil}(0,T;L^{2}(\mathcal O_0)),
\end{equation}
in particular for $q$ solution to \eqref{eq:forward_linear}
\begin{equation*}
\mathcal{L}(q) = - \mathbb E \int_{\mathcal O} y_T q(T) dx = \mathbb E \int_{0}^{T} \int_{\mathcal O_0} h q dx dt\qquad \forall q_0 \in L^2_{\F_0}(\Omega;L^{2}(\mathcal O)).
\end{equation*}
As before we easily deduce by Itô's formula that $y(0) = 0$ in $\mathcal O$ a.s.

Our goal is then to prove that $h \in L_{\fil}^{\infty}(0,T;L^{\infty}(\mathcal O_0))$ and
\begin{equation}
\label{eq:boundhlinftyproof}
\|h\|_{L_{\fil}^{\infty}(0,T;L^{\infty}(\mathcal O_0))} \leq   2 C \|y_T\|_{L^{\infty}_{\F_T}(\Omega;L^{\infty}(\mathcal O))}.
\end{equation}
From \eqref{eq:normfunctionalL1} and \eqref{eq:equationcontrolFunctionalL2}, we deduce that
\begin{equation}
\label{eq:EstimationObservabilityWithL}
\left| \mathbb E \int_{0}^{T} \int_{\mathcal O_0} h q dx dt \right| \leq C \|q\|_{L^1_{\fil}(0,T;L^1(\mathcal O))} \|y_T\|_{L^{\infty}_{\F_T}(\Omega;L^{\infty}(\mathcal O))} \quad \forall q \in L^{2}_{\fil}(0,T;L^{2}(\mathcal O_0)).
\end{equation}
Then we take
\begin{equation*}
q = \text{sign}(h) \mathbf{1}_{\{|h| > 2 C \|y_T\|_{L^{\infty}_{\F_T}(\Omega;L^{\infty}(\mathcal O))}\}} \mathbf{1}_{\mathcal O_0} ,
\end{equation*}
so the left hand side of \eqref{eq:EstimationObservabilityWithL} is bounded from below by
\begin{equation*}
2 C \|y_T\|_{L^{\infty}_{\F_T}(\Omega;L^{\infty}(\mathcal O))} \mu_{\Omega, [0,T], \mathcal O_0}(|h| > 2 C \|y_T\|_{L^{\infty}_{\F_T}(\Omega;L^{\infty}(\mathcal O))}),
\end{equation*}
while the right hand side of \eqref{eq:EstimationObservabilityWithL} translates into
\begin{equation*}
C \|y_T\|_{L^{\infty}_{\F_T}(\Omega;L^{\infty}(\mathcal O))} \mu_{\Omega, [0,T], \mathcal O_0}(|h| > 2 C \|y_T\|_{L^{\infty}_{\F_T}(\Omega;L^{\infty}(\mathcal O))}),
\end{equation*}
where $\mu_{\Omega, [0,T], \mathcal O_0}$ denotes the product measure on $\Omega \times [0,T] \times \mathcal O$, so we get
\begin{multline*}
2 C \|y_T\|_{L^{\infty}_{\F_T}(\Omega;L^{\infty}(\mathcal O))} \mu_{\Omega, [0,T], \mathcal O_0}(|h| > 2 C \|y_T\|_{L^{\infty}_{\F_T}(\Omega;L^{\infty}(\mathcal O))})\\
 \leq C  \|y_T\|_{L^{\infty}_{\F_T}(\Omega;L^{\infty}(\mathcal O))}\mu_{\Omega, [0,T], \mathcal O_0}(|h| > 2 C \|y_T\|_{L^{\infty}_{\F_T}(\Omega;L^{\infty}(\mathcal O))}).
\end{multline*}
That is a contradiction unless we have
\begin{equation*}
\mu_{\Omega, [0,T], \mathcal O_0}(|h| > 2 C \|y_T\|_{L^{\infty}_{\F_T}(\Omega;L^{\infty}(\mathcal O))}) = 0.
\end{equation*}
This exactly means that we have
\begin{equation*}
|h| \leq  2 C \|y_T\|_{L^{\infty}_{\F_T}(\Omega;L^{\infty}(\mathcal O))} \quad \text{a.e. }(t,x)\in Q_T,\ \text{a.s.}
\end{equation*}
This concludes the proof of \eqref{eq:boundhlinftyproof}, then the one of \Cref{thm:control}.
\end{proof}

\subsection{The nonlinear control problem}

The proof of \Cref{theo:control_semi} relies on the so-called source term method, originally developed in \cite{LLT13} for the deterministic setting and later extended to SPDEs in \cite{HSLBP20a}. While \cite{HSLBP20a} already discusses this method for BSPDEs, here we provide a more comprehensive treatment and extend the analysis to our setting.

Let us consider a slightly more general version of the linear equation \eqref{eq:backward_linear}, namely
\begin{equation}\label{eq:backward_linear_source}
\begin{cases}
\d{y}=-\left(\Delta y + \alpha y + \chi_{\dom_0}h+S\right)\d{t}+Y\d{W}(t) & \text{in } Q_T, \\
y=0 & \text{on } \Sigma_T, \\
y(T)=y_T & \text{in } \dom,
\end{cases}
\end{equation}
where $S$ is an external source term that will be taken from a suitable functional space and $h$ is a control.

From \Cref{thm:new_space}, we can readily deduce that if the control $h\in L^\infty_{\fil}(0,T;L^\infty(\dom_0))$ and the source term $S\in L^\infty_{\fil}(0,T;L^\infty(\dom))$, then \eqref{eq:backward_linear_source} is well-posed and $(y,Y)\in\mathscr Y_p$ for any $p\in[2,+\infty)$ (recall \eqref{eq:Yp_space}). On the other hand, \Cref{thm:control} already tells us that if $S\equiv 0$, then \eqref{eq:backward_linear_source} is null-controllable in $L^\infty$. Our first aim is to see that if $S$ belongs to an appropriate space, then \eqref{eq:backward_linear_source} remains null-controllable in $L^\infty$.

To this end, we fix $M \geq 1$ such that 
\begin{equation}\label{cost_M}
e^{C/T} \leq M e^{M/T},
\end{equation}
where $C$ is the constant provided by \Cref{thm:control} (with $p=+\infty$). For some parameters $Q \in (1,\sqrt{2})$ and $R > Q^2/(2-Q^2)$, we define
\begin{align}\label{weight_source}
\rho(t)&:=\exp\left(-\frac{(1+R)Q^2 M}{(Q-1)t}\right), \\ \label{weight_control} 
\rho_0(t)&:=\exp\left(-\frac{R M}{(Q-1)t}\right),
\end{align}
which are strictly increasing and vanish as $t \to 0^+$. 

Using these functions, we introduce the following weighted spaces
\begin{align}
\mathscr S&:=\left\{S\in L^\infty_{\fil}(0,T;L^\infty(\dom)):\ \frac{S}{\rho}\in L^\infty_{\fil}(0,T;L^\infty(\dom))\right\}, \\
\mathscr E&:=\left\{y\in L^\infty_{\fil}(0,T;L^\infty(\dom)):\ \frac{y}{\rho_0}\in L^\infty_{\fil}(0,T;L^\infty(\dom))\right\}, \\
\mathscr H&:=\left\{h\in L^\infty_{\fil}(0,T;L^\infty(\dom_0)):\ \frac{h}{\rho_0}\in L^\infty_{\fil}(0,T;L^\infty(\dom_0))\right\}.
\end{align}

We have the following result. 

\begin{proposition}\label{prop:source}
For every $S\in\mathscr S$ and $y_T\in L^\infty_{\mathcal F_T}(\Omega;L^\infty(\Omega))$, there is a control $h\in \mathscr H$ such that $y$ the first component of the solution $(y,Y)\in \mathscr Y_p$ to \eqref{eq:backward_linear_source} satisfies $y\in \mathscr E$. Moreover, the exists a constant $C>0$ independent of $S$ and $y_T$ such that 
\begin{align}\notag 
&\|y/\rho_0\|_{L^\infty_{\fil}(0,T;L^\infty(\dom))} + \|h/\rho_0\|_{L^\infty_{\fil}(0,T;L^\infty(\dom))} + \|Y\|_{L^p_{\fil}(\Omega;L^2(0,T;L^2(\dom)))}\\ \label{eq:linf_est_source}
&\qquad \leq Ce^{C /T}  \left( \norme{y_T}_{L^{\infty}_{\mathcal F_T}(\Omega;L^{\infty}(\mathcal O))} + \norme{S/\rho}_{L^{\infty}_{\fil}(0,T;L^{\infty}(\dom))}\right),
\end{align}
\end{proposition}
\begin{proof}
The proof follows the same steps of \cite[Proposition 2.5]{HSLBP20a} with some adaptations to the backward case and the $L^\infty$-setting (cf. \cite{LB20}). 

Let $T_k:=TQ^{-k}$ where we recall that $Q\in (1,\sqrt{2})$ is a given parameter. Let $a_0:=y_T$ and for $k\in\mathbb N$ we define $a_{k+1}:=y_S(T_{k+1},\cdot)$ where $y_S$ is the first component of $(y_S,Y_S)$ the solution to the uncontrolled system
\begin{equation}\label{eq:backward_linear_source_split_S}
\begin{cases}
\d{y_S}=-\left(\Delta y_S + \alpha y_S +S\right)\d{t}+Y_S\d{W}(t) & \text{in } (T_{k+1},T_k)\times\dom, \\
y_S=0 & \text{on } (T_{k+1},T_k)\times\dom, \\
y_S(T_k)=0 & \text{in } \dom,
\end{cases}
\end{equation}
We observe that by \Cref{thm:new_space}, $y_S$ is a continuous adapted function in $[T_{k+1},T_k]$, $a_{k+1}$ is $\mathcal F_{T_{k+1}}$ measurable and belongs to $L^\infty(\Omega\times\dom)$. In particular, by estimate \eqref{est:Yp_full} 
\begin{equation}\label{eq:est_akp1}
\|a_{k+1}\|_{L^\infty(\Omega \times \dom)}\leq \|y\|_{L^\infty_{\fil}(T_{k+1},T_k;L^\infty(\dom))} \leq \|(y,Y)\|_{\mathscr Y_p} \leq C\|S\|_{L^\infty_{\fil}(T_{k+1},T_k;L^\infty(\dom))}.
\end{equation}

Let $\{a_k\}_{k\in \mathbb{N}}$ be the sequence defined by the above procedure. For each $k\in\mathbb{N}$, consider the controlled equation (without a source term) given by
\begin{equation}\label{eq:backward_linear_source_split_H}
\begin{cases}
\d{y_H}=-\left(\Delta y_H + \alpha y_H + \chi_{\dom_0}h_k\right)\d{t} + Y_H\d{W}(t) & \text{in } (T_{k+1},T_k)\times\dom, \\
y_H=0 & \text{on } (T_{k+1},T_k)\times\partial\dom, \\
y_H(T_k)=a_k & \text{in } \dom.
\end{cases}
\end{equation}
This system is well-posed (in the sense of \Cref{thm:new_space}) for each $k\in\mathbb{N}$ due to the properties of $a_k$ mentioned above. In view of \Cref{thm:control}, we can obtain controls $h_k\in L^\infty_{\fil}(T_{k+1},T_k;L^\infty(\dom_0))$ such that $y_H(T_{k+1})=0$. Moreover, using \eqref{cost_M} and the precise estimate \eqref{eq:EsimationControlhLp}, it can be deduced that
\begin{equation}\label{est_control_k}
\|h_k\|_{L^\infty_{\fil}(T_{k+1},T_k;L^\infty(\dom_0))} \leq M e^{\frac{M}{T_{k}-T_{k+1}}} \|a_{k}\|_{L^\infty_{\mathcal F_{T_k}}(\Omega;L^\infty(\dom))}, \qquad \forall k\in\mathbb{N}.
\end{equation}
In particular, for $k=0$, we have
\begin{equation}
\|h_0\|_{L^\infty_{\fil}(T_{1},T;L^\infty(\dom_0))} \leq M e^{\frac{MQ}{T(Q-1)}} \|y_T\|_{L^\infty_{\mathcal F_{T}}(\Omega;L^\infty(\dom))},
\end{equation}
and, from the definition of $\rho_0$ (see \eqref{weight_control}), we see that $\rho_0(t_1)<\rho_0(t_2)$ for every $0\leq t_1<t_2\leq T$, whence
\begin{equation}\label{eq:est_control_0}
\|h_0/\rho_0\|_{L^\infty_{\fil}(T_{1},T;L^\infty(\dom_0))} \leq \rho_0^{-1}(T_1)M e^{\frac{MQ}{T(Q-1)}} \|y_T\|_{L^\infty_{\mathcal F_{T}}(\Omega;L^\infty(\dom))}
\end{equation}

On the other hand, putting together \eqref{eq:est_akp1} and \eqref{est_control_k} yields
\begin{equation}
\|h_{k+1}\|_{L^\infty_{\fil}(T_{k+2},T_{k+1};L^\infty(\dom))}\leq CMe^{\frac{M}{T_{k+1}-T_{k+2}}}\|S\|_{L^\infty_{\fil}(T_{k+1},T_k;L^\infty(\dom))}
\end{equation}
Similar as above, from \eqref{weight_source} we have $\rho(t_1)<\rho(t_2)$ for every $0\leq t_1<t_2\leq T$ and therefore
\begin{equation}
\|h_{k+1}\|_{L^\infty_{\fil}(T_{k+2},T_{k+1};L^\infty(\dom))}\leq CMe^{\frac{M}{T_{k+1}-T_{k+2}}}\rho(T_k)\|S/\rho\|_{L^\infty_{\fil}(T_{k+1},T_k;L^\infty(\dom))}.
\end{equation}
We note that by definitions \eqref{weight_source} and \eqref{weight_control}, the identity $e^{\frac{M}{T_{k+1}-T_{k+2}}}\rho(T_k)=\rho_0(T_{k+2})$ holds. Thus,
\begin{equation}
\|h_{k+1}\|_{L^\infty_{\fil}(T_{k+2},T_{k+1};L^\infty(\dom))}\leq CM\rho_0(T_{k+2})\|S/\rho\|_{L^\infty_{\fil}(T_{k+1},T_k;L^\infty(\dom))},
\end{equation}
and since $\rho_0(T_{k+2})\leq \rho(t)$ for all $t\in(T_{k+2},T_{k+1})$
\begin{equation}\label{eq:est_control_source}
\|h_{k+1}/\rho_0\|_{L^\infty_{\fil}(T_{k+2},T_{k+1};L^\infty(\dom))}\leq CM\|S/\rho\|_{L^\infty_{\fil}(T_{k+1},T_k;L^\infty(\dom))}, \qquad \forall k\in\mathbb N.
\end{equation}

Let us define $h(t,\cdot):=\sum_{k\in\mathbb{N}}1_{\left(T_{k+1},T_{k}\right)}(t)h_k(t,\cdot)$, where each control $h_k$ satisfies \eqref{eq:est_control_source} for $k\geq 1$ and \eqref{eq:est_control_0} if $k=0$. Defined in this way, since each interval $(T_{k+1},T_k)$ is disjoint and $h_k$ are independent processes for each $k\in\mathbb{N}$, we have that $h\in L^\infty_{\fil}(0,T;L^\infty(\dom))$. Moreover, by \eqref{eq:est_control_0}, \eqref{eq:est_control_source}, and using the fact that the essential supremum over disjoint intervals equals the supremum of the individual norms, we get
\begin{align}\notag
\|h/\rho_0\|_{L^\infty_{\fil}(0,T;L^\infty(\dom))}
&=\sup_{k\in\mathbb{N}} \|h_k/\rho_0\|_{L^\infty_{\fil}(T_{k+1},T_k;L^\infty(\dom))} \\ \notag
&\leq \rho_0^{-1}(T_1) M e^{\frac{MQ}{T(Q-1)}} \|y_T\|_{L^\infty_{\mathcal{F}_{T}}(\Omega;L^\infty(\dom))}
+ C M \sup_{k\in\mathbb{N}} \|S/\rho\|_{L^\infty_{\fil}(T_{k+1},T_k;L^\infty(\dom))} \\ \label{est:control_sum}
&= C^\prime e^{C^\prime/T} \|y_T\|_{L^\infty_{\mathcal{F}_{T}}(\Omega;L^\infty(\dom))} + C^\prime \|S/\rho\|_{L^\infty_{\fil}(0,T;L^\infty(\dom))}
\end{align}
for some constant $C^\prime>0$ uniform with respect to $y_T$ and $S$.

In a similar way, we can build the solution to \eqref{eq:backward_linear_source} by gluing together $y_S$ and $y_H$, the solutions to \eqref{eq:backward_linear_source_split_S} and \eqref{eq:backward_linear_source_split_H}, respectively. Indeed, from \Cref{thm:new_space} and recalling \eqref{eq:Yp_space}, we have for each time interval $(T_{k+1},T_k)$
\begin{equation}\label{est_ys_inter}
\|y_S\|_{L^\infty_{\fil}(T_{k+1},T_k;L^\infty(\dom))} \leq C \|S\|_{L^\infty_{\fil}(T_{k+1},T_k;L^\infty(\dom))}
\end{equation}
and
\begin{equation}\label{est_yh_inter}
\|y_H\|_{L^\infty_{\fil}(T_{k+1},T_k;L^\infty(\dom))} \leq C \left(\|a_k\|_{L^\infty_{\mathcal{F}_{T_k}}(\Omega;L^\infty(\dom))} + \|h_k\|_{L^\infty_{\fil}(T_{k+1},T_k;L^\infty(\dom))}\right).
\end{equation}
We note that, by construction, $y(t,\cdot)=y_S(t,\cdot)+y_H(t,\cdot)$ is a continuous adapted function in each subinterval $(T_{k+1},T_{k})$ and, moreover, it is continuous at the junction times $T_k$, $k\in\mathbb{N}$. Furthermore, we can deduce from \eqref{eq:est_control_0}, \eqref{eq:est_control_source}, \eqref{est_ys_inter}, and \eqref{est_yh_inter}, using an argument similar to that for the control, that
\begin{align}\notag
\|y/\rho_0\|_{L^\infty_{\fil}(0,T;L^\infty(\dom))}
&\leq \|y_H/\rho_0\|_{L^\infty_{\fil}(0,T;L^\infty(\dom))} + \|y_S/\rho_0\|_{L^\infty_{\fil}(0,T;L^\infty(\dom))} \\ \notag
&= \sup_{k\in\mathbb{N}} \left( \|y_H/\rho_0\|_{L^\infty_{\fil}(T_{k+1},T_k;L^\infty(\dom))} + \|y_S/\rho_0\|_{L^\infty_{\fil}(T_{k+1},T_k;L^\infty(\dom))} \right) \\ \notag
&\leq C \rho_0^{-1}(T_1) M e^{\frac{MQ}{T(Q-1)}} \|y_T\|_{L^\infty_{\mathcal{F}_{T}}(\Omega;L^\infty(\dom))}
+ C M^{-1} \sup_{k\in\mathbb{N}} \|S/\rho\|_{L^\infty_{\fil}(T_{k+1},T_k;L^\infty(\dom))} \\ \label{est_sum_estate}
&= C^\prime e^{C^\prime/T} \|y_T\|_{L^\infty_{\mathcal{F}_{T}}(\Omega;L^\infty(\dom))} + C^{\prime\prime} \|S/\rho\|_{L^\infty_{\fil}(0,T;L^\infty(\dom))}.
\end{align}
for some constant $C^{\prime\prime}>0$ independent of $S$. 

Estimate \eqref{eq:linf_est_source} follows from \eqref{est:control_sum}, \eqref{est_sum_estate}, and by choosing an appropriate constant $C>0$, which is uniform with respect to $S$ and $y_T$. To add the estimate on $Y$ in \eqref{eq:linf_est_source} we use \eqref{est:Yp_full} and the fact that $\rho$ is bounded in $[0,T]$. This ends the proof.
\end{proof}

\begin{rmk}
If the component $y$ of the solution $(y,Y)\in\mathscr Y_p$ of \eqref{eq:backward_linear_source} satisfies $y\in\mathscr E$, then $y(0,\cdot)=0$ in $\Omega$ a.s.
\end{rmk}

We are now in position to prove our main controllability result. 

\begin{proof}[Proof of \Cref{theo:control_semi}]
Let us define the following Banach space
$$ \mathscr Z_p = \mathscr{E} \times L^{p}_{\fil}(\Omega;L^2(0,T;L^2(\dom))).$$

By Taylor's formula at second order, we have that
\begin{equation}\label{eq:taylorcontrol}
\forall s\in\mathbb R, \quad f(s)=f^\prime(0)s+s^2G(s), \quad 
\end{equation}
where $G(s)=\int_0^1 (1-\sigma) f^{\prime\prime}(\sigma s)\d{\sigma}$. For any $r\in(0,1)$, let 
\begin{equation}\label{Y_rcontrol}
\mathscr Z^r_p:=\{(y,Y)\in\mathscr Z_p: \|y\|_{\mathscr E} + \|Y\|_{L^p_{\fil}(\Omega;L^2(0,T;L^2(\dom)))}\leq r\},
\end{equation} and consider the linear controlled equation
\begin{equation}
\label{eq:linearizedControl}
\begin{cases}
\d{y}=-\left(\Delta y + \alpha y + f^\prime(0) y + \ov{y}^2G(\ov y) + \chi_{\mathcal{O}_0} h \right)\d{t}+Y\d{W}(t) &\text{in } Q_T, \\
y=0 &\text{on } \Sigma_T, \\
y(T)=y_T &\text{in }\dom,
\end{cases}
\end{equation}
where $\ov y$ is the first entry of any given pair $(\ov y, \ov Y)\in\mathscr Z_p^r$ and $y_T\in L^\infty_{\F_T}(\Omega;L^\infty(\dom))$. 

Let us define the following multi-valued contraction mapping
\begin{equation}
\mathcal{N} : (\overline{y}, \overline{Y}) \in \mathscr Z_p^r \to 2^{\mathscr Z_p^r},
\end{equation}
such that for  every $(\overline{y}, \overline{Y}) \in \mathscr Z_p^r$, $(y,Y)$ belongs to $\mathcal{N} ((\overline{y}, \overline{Y}))$, if there exists $h \in \mathscr H$ such that $(y,Y) \in \mathscr Z_p^r$ is the solution of \eqref{eq:linearizedControl} and estimate \eqref{eq:linf_est_source} holds.

Since $f\in C^\infty(\mathbb R)$, from point \ref{eq:est_sec} of \Cref{lem:nonlinearity}, we have that $|G(s)|\leq M$ for all $s\in[-2,2]$ where the constant $M>0$ only depends on $f$. Moreover, we have from the definitions of the weights $\rho$ and $\rho_0$ that there exists $C>0$ such that $|(\rho_0^2 / \rho)(t)| \leq C$ for every $t\in[0,T]$. Hence, defining $\ov S:=\ov y^2G(\ov y)$ we have
\begin{align}\notag 
\|\ov{S}/\rho\|_{L^\infty_{\fil}(0,T;L^\infty(\dom))}&=\|\ov y^2G(\ov y)/\rho\|_{L^\infty_{\fil}(0,T;L^\infty(\dom))} \\ \notag
&\leq C M \|\ov{y}^2 /\rho_0^2\|_{L^\infty_{\fil}(0,T;L^\infty(\dom))} \\
\label{eq:est_infty_poids}
&\leq C  M\|\ov y/\rho_0\|_{L^\infty_{\fil}(0,T;L^\infty(\dom))}^2\leq C M r^2 < +\infty,
\end{align}
for any $(\ov y, \ov Y)\in\mathscr Z_p^r$. 

Thus, using \Cref{prop:source} with this particular $\ov{S}$ and estimate \eqref{eq:linf_est_source} yield that there is a controlled trajectory solution $((y,Y), h)\in\mathscr Z_p \times \mathscr{H}$ to \eqref{eq:linearizedControl} such that 
\begin{align}\label{eq:est_y_control}
\|y\|_{\mathscr E} + \|h\|_{\mathcal{H}} + \|Y\|_{L^p_{\fil}(\Omega;L^2(0,T;L^2(\dom)))} &\leq C\left(\|y_T\|_{L^\infty_{\F_T}(\Omega;L^\infty(\dom))}+Mr^2\right)
\end{align}
for some $C>0$ only depending on $\dom$, $T$ and $p$. Let us take 
\begin{equation*}
r\leq \frac{1}{8CM},
\end{equation*}
 where $C$ is the constant appearing in \eqref{eq:est_y_control} and fix $\delta>0$ such that $\delta \leq r/2C$. Thus for any $y_T\in L^\infty_{\F_T}(\Omega;L^\infty(\dom))$ verifying $\|y_T\|_{L^\infty_{\F_T}(\Omega;L^\infty(\dom))}\leq \delta$ we have
\begin{equation}\label{eq:est_y_r_control}
\|y\|_{\mathscr E} + \|h\|_{\mathcal{H}} + \|Y\|_{L^p_{\fil}(\Omega;L^2(0,T;L^2(\dom)))}\leq \frac{r}{2}+\frac{r}{8}=  \frac{5r}{8}\leq r.
\end{equation}
This shows that for every $(y,Y) \in \mathscr Z_p^r$,  $\mathcal N(y,Y)$ is nonempty. 

Moreover, for every $(y,Y) \in \mathscr Z_p^r$, from the estimate \eqref{eq:linf_est_source} it is straightforward to check that that $\mathcal N(y,Y)$ is a bounded and closed set of $\mathscr Z_p^r$.

To conclude, we will verify that $\mathcal N$ is a multi-valued contraction mapping on the Banach space $\mathscr Z_p^r$ and use the multi-valued Banach fixed point theorem from \cite[Theorem 5]{Nad69}. By point \ref{eq:est_dif_G} of \Cref{lem:nonlinearity} with $\mathcal I=[-2,2]$, we have that for  any $\ov y_1$, $\ov y_2$ taken from the pairs $(\ov{y}_{i},\ov{Y}_{i})\in\mathscr Z_p^r$, $i=1,2$,
\begin{align}\notag
&\|(\ov{y}_1G(\ov y_1)-\ov{y}_2 G(\ov y_2))/\rho\|_{L^\infty_{\fil}(0,T;L^\infty(\dom))}\\ \notag
&\qquad \leq C M \|(\ov{y}_1-\ov{y}_2)/\rho_0\|_{L^\infty_{\fil}(0,T;L^\infty(\dom))}\left(\|\ov{y}_1/\rho_0\|_{L^\infty_{\fil}(0,T;L^\infty(\dom))}+\|\ov{y}_2/\rho_0\|_{L^\infty_{\fil}(0,T;L^\infty(\dom))}\right) \\ \label{est:rhs_control}
&\quad \qquad +C  M_1\|y_2/\rho_0\|^2_{L^\infty_{\fil}(0,T;L^\infty(\dom))} \|(\ov{y}_1-\ov{y}_2)/\rho_0\|_{L^\infty_{\fil}(0,T;L^\infty(\dom))},
\end{align}
for some positive constants $M,M_1$ only depending on $f$. Above, we have used once again that $|(\rho_0^2 / \rho)(t)| \leq C$ for every $t\in[0,T]$.

Denote by $(y_i,Y_i) \in \mathcal{N}(\ov{y_i},\ov{Y_i})$ the solution of \eqref{eq:linearizedControl} for the corresponding $\ov{y}_i$, $i=1,2$. From the linearity of \eqref{eq:linearizedControl}, \Cref{prop:source} with $\ov{S}=\ov{y}_1G(\ov y_1)-\ov{y}_2 G(\ov y_2)$ and $y_T\equiv 0$, and estimate \eqref{est:rhs_control}, we can compute
\begin{align}\notag
&\|\mathcal ({y}_1,{Y}_1)- ({y}_2,{Y}_2)\|_{\mathscr Z_p} \\ \notag
&\qquad \leq C M \|(\ov{y}_1-\ov{y}_2)/\rho_0\|_{L^\infty_{\fil}(0,T;L^\infty(\dom))}\left(\|\ov{y}_1/\rho_0\|_{L^\infty_{\fil}(0,T;L^\infty(\dom))}+\|\ov{y}_2/\rho_0\|_{L^\infty_{\fil}(0,T;L^\infty(\dom))}\right)  \\ \notag
&\quad \qquad + C  M_1\|y_2/\rho_0\|^2_{L^\infty_{\fil}(0,T;L^\infty(\dom))} \|(\ov{y}_1-\ov{y}_2)/\rho_0\|_{L^\infty_{\fil}(0,T;L^\infty(\dom))} \\
& \qquad \leq 2CM r \|(\ov{y}_1-\ov{y}_2)/\rho_0\|_{L^\infty_{\fil}(0,T;L^\infty(\dom))}
 + CM_1 r^2 \|(\ov{y}_1-\ov{y}_2)/\rho_0\|_{L^\infty_{\fil}(0,T;L^\infty(\dom))}.\label{eq:est_r_control}
\end{align}
From \eqref{eq:est_r} and decreasing (if necessary) the value of $r$ and taking $r\leq \min\{\frac{1}{8CM},\frac{1}{\sqrt{4CM_1}}\}$, we get
\begin{align*}
\|\mathcal ({y}_1,{Y}_1)- ({y}_2,{Y}_2)\|_{\mathscr Z_p} \leq \frac{1}{2}\|(\ov{y}_1-\ov{y}_2)/\rho_0\|_{L^\infty_{\fil}(0,T;L^\infty(\dom))}\leq  \frac{1}{2} \| (\ov y_1,\ov Y_1)-(\ov y_2,\ov Y_2)\|_{\mathscr Z_p}.
\end{align*}
Thus $\mathcal N$ is a multi-valued contraction mapping on the Banach space $\mathscr Z_p^r$. Hence there is a fixed point for $\mathcal N$ which is a weak solution to \eqref{eq:backward_nonlinear} thanks to \eqref{eq:taylorcontrol} and that belongs to $\mathscr Z_p^r$ since  \eqref{eq:est_y_r_control} holds. 
\end{proof}

\begin{rmk}
The following comments are in order.

\begin{itemize}
\item In the proof of \Cref{theo:control_semi}, we employed a multi-valued fixed point theorem, in contrast to the standard Banach fixed point argument used in \Cref{thm:semi}. The main reason for this is that \Cref{prop:source} guarantees only the existence of a control that steers the solution of \eqref{eq:backward_linear_source} to zero at $t=0$. Similar strategies based on multi-valued fixed point theorems have been used in the control literature; see, for instance, the use of Kakutani's fixed point theorem in \cite{FCZ00}.
\item On the other hand, although the control produced in \Cref{theo:control_semi} is not unique, the corresponding controlled trajectory of \eqref{eq:backward_nonlinear} is. Once a control steering the solution to zero is fixed, uniqueness follows by an argument similar to the one at the end of the proof of \Cref{thm:semi}.
\end{itemize}
\end{rmk}

{
\section{On the optimality of the regularity}\label{new_sec_reg}

In this section, we provide additional remarks on the regularity results established in Theorems \ref{thm:ito} and \ref{thm:infinity}. By combining these results, we deduce that if $y_T\in L^\infty_{\mathcal F_T}(\Omega;L^\infty(\dom))$ and $F\in L^\infty_{\fil}(0,T;L^\infty(\dom))$, then the solution to \eqref{eq:intro} satisfies
\begin{equation}
(y,Y)\in L^\infty_{\fil}(0,T;L^\infty(\dom))\times L^p_{\fil}(\Omega;L^2(0,T;L^2(\dom)))
\end{equation}
for every $p\ge 2$. However, from the proof of \Cref{thm:infinity}, which itself relies on \Cref{prop:linf_uniform}, it is not evident whether the regularity of $Y$ (in $\Omega$) can be improved to $L^\infty$. In fact, the term involving $Y$ was explicitly removed in the derivation of the estimates (see, for example, inequality \eqref{est:lp_ft}). Therefore, the purpose of this section is to present an example showing that such an improvement is, in general, not possible in the simpler case of backward ordinary stochastic  differential equations (BSDE).

To this end, let $T>0$ be a fixed deterministic time and consider the BSDE
\begin{equation}\label{eq:sode}
\begin{cases}
\d y = \lambda y \dt + Y \d W(t) \quad t\in(0,T), \\
y(T) = y_T,
\end{cases}
\end{equation}
where $\lambda> 0$ is a constant. Note that this equation can be interpreted as a projection of the linear stochastic system \eqref{eq:intro} onto a suitable subspace of $L^2$.

By standard well-posedness results for BSDEs (see, for instance, \cite[Chapter 7, Theorem 2.2]{YZ99}), equation \eqref{eq:sode} admits a unique adapted solution
\begin{equation}\label{reg_yY}
(y,Y)\in L^2_{\fil}(\Omega;C([0,T];\mathbb R))\times L^2_{\fil}(0,T;\mathbb R).
\end{equation}
We will show that one can choose a particular terminal condition $y_T\in L^\infty_{\mathcal F_T}(\Omega;\mathbb R)$ such that the corresponding solution satisfies $y\in L^\infty_{\fil}(0,T;\mathbb R)$
while ensuring that $Y\notin L^\infty_{\mathcal F}(\Omega;L^2(0,T;\mathbb R))$.

Consider $y_T=\mathbf{1}_{\{W_T\geq c\}}$ where $c>0$ is a fixed number and where we use the notation $W_{\cdot}:=W(\cdot)$. Note that by construction $y_T\in L^\infty_{\mathcal F_T}(\Omega;\mathbb R)$. Moreover, by using the change of variables $(z,Z):=e^{-\lambda t}(y,Y)$ and It\^{o}'s formula, equation \eqref{eq:sode} becomes
\begin{equation}\label{eq:sode_z}
\begin{cases}
\d z = Z\,\d W(t) \quad t\in(0,T), \\
z(T) = e^{-\lambda T}\mathbf{1}_{\{W_T\geq c\}}.
\end{cases}
\end{equation}
By the martingale representation theorem, it is not difficult to see that the process
\begin{equation}\label{sol_z_explicit}
z(t)=\E\left( e^{-\lambda T}\mathbf{1}_{\{W_T\geq c\}} \bigmid \mathcal F_t\right), \quad t\in(0,T),
\end{equation}
solves equation \eqref{eq:sode_z} (together with its initial condition). Moreover, by the change of variable $y=e^{\lambda t} z$, $y(T)=\mathbf{1}_{\{W_T\geq c\}}$ and by independence of the Brownian increments
\begin{align*}
y(t)&=e^{\lambda t}\E\left( e^{-\lambda T}\mathbf{1}_{\{W_T\geq c\}} \bigmid \mathcal F_t\right) = e^{-\lambda(T-t)}\mathbb P\left(W_{T-t}\geq c-W_t\right) \\
&= e^{-\lambda(T-t)} \int_{\mathbb R}\mathbb P\left(W_{T-t}\geq c-x\bigmid W_t=x\right)\mathbb P(W_t\in \d x),
\end{align*}
whence,
\begin{equation}\label{eq:y_exp}
y(t)=e^{-\lambda(T-t)}\Phi\left(\frac{c-W_t}{\sqrt{T-t}}\right),
\end{equation}
where $\Phi(x):=\displaystyle\int_{x}^{\infty}\frac{e^{-u^2/2}}{\sqrt{2\pi}}\d{u}$. Since $\|\Phi\|_{L^\infty(\mathbb R)}=\frac{1}{\sqrt{2\pi}}$ it follows that $y\in L^\infty (\Omega;L^\infty(0,T;\mathbb R))$. Moreover, $z$ is  adapted to $\fil$ by construction (see \eqref{sol_z_explicit}),  hence $y=e^{\lambda t}z$ is adapted as well, and therefore $y\in L^\infty_{\fil}(0,T;\mathbb R)$. 

Let us define $v(t):=\frac{c-W_t}{\sqrt{T-t}}$. Applying It\^{o}'s formula to \eqref{eq:y_exp} yields after a straightforward computation that
\begin{equation}
\d y(t)=\lambda e^{-\lambda(T-t)}\Phi(v(t))\dt - \frac{e^{-\lambda(T-t)}}{\sqrt{T-t}}\Phi^\prime(v(t))\d{W(t)}.
\end{equation}
Comparing this with \eqref{eq:sode} and \eqref{eq:y_exp}, we identify for $t\in(0,T)$
\begin{equation}\label{Y_iden}
Y(t)=  \frac{1}{\sqrt{2\pi}}\frac{e^{-\lambda(T-t)}}{\sqrt{T-t}} \exp\left(-\frac{1}{2}\left(\frac{c-W_t}{\sqrt{T-t}}\right)^2\right), \quad \text{a.s.}  
\end{equation}

\begin{proposition}\
The process $Y$ defined \eqref{Y_iden} does not belong to $ L^\infty_\fil(\Omega;L^2(0,T;\mathbb R))$.
\end{proposition}

\begin{proof}
Let us define the function 
\begin{equation*}
\omega\in \Omega \mapsto I(\omega):=\int_{0}^{T}\frac{1}{2\pi}\frac{e^{-2\lambda(T-t)}}{{T-t}} \exp\left(-\left(\frac{c-W_t(\omega)}{\sqrt{T-t}}\right)^2\right)\dt.
\end{equation*}
Note that $I(\omega)\geq 0$ and from \eqref{reg_yY}, $I(\omega)<+\infty$ for almost all $\omega\in \Omega$. So, to prove that $Y\notin L^\infty(\Omega;L^2(0,T))$, it is enough to show that for every given constant $M>0$
\begin{equation*}
\mathbb P\left(I\geq M\right)>0.
\end{equation*}

Let $\delta\in(0,T/4)$ be a given and consider the intervals
\begin{equation*}
I_j=\left[T-2^{-j}\delta, T-2^{-j-1}\delta\right]
\end{equation*}
so that $\cup_{j\in\mathbb N} I_j=[T-\delta, T]$. We note that $|I_j|=\delta 2^{-j-1}$. Let $\epsilon>0$ be a fixed quantity and  define the events
\begin{equation*}
A_j=\left\{ \sup_{t\in I_j}\left|W_t-W_{T-\delta 2^{-j}}\right|\leq \epsilon\sqrt{|I_j|} \quad\text{and}\quad |W_{T-\delta 2^{-j}}-c|\leq \epsilon \sqrt{|I_j|} \right\}.
\end{equation*}
We have the following intermediate result.

\begin{lemma}\label{claim}  
Let $j_0\in\mathbb N$ be large enough. Define $A^N:= \cap_{j=j_0}^{N} A_j$ for some natural number $N>j_0$. Then 
\begin{equation}\label{prodN}
\mathbb P(A^N) >0.
\end{equation}
\end{lemma}

We postpone the proof of this result to the end of this section. Assuming that Lemma \ref{claim} holds, we have that 
\begin{align}\label{est_I}
I(\omega) &\geq \frac{e^{-2\lambda T}}{2\pi}  \int_{\cup_{j=j_0}^{N}I_j} \frac{1}{{T-t}} \exp\left(-\left(\frac{c-W_t(\omega)}{\sqrt{T-t}}\right)^2\right)\dt.
\end{align}
Moreover, if $\omega\in A^N$, then for all $t\in I_j$ with $j\in\{j_0,\ldots,N\}$
\begin{equation}\label{est_Wt-c}
|W_t(\omega)-c|\leq \sup_{t\in I_j}\left|W_t(\omega)-W_{T-\delta 2^{-j}}(\omega)\right|+|W_{T-\delta 2^{-j}}(\omega)-c| \leq 2\epsilon \sqrt{|I_j|},
\end{equation}
and
\begin{equation}\label{est_T-t}
(T-t)^{-1}\leq 2^{j+1}\delta^{-1}.
\end{equation}
Putting together estimates \eqref{est_I} to \eqref{est_T-t}, we get
\begin{equation*}
I(\omega)\geq \frac{e^{-2\lambda T}}{2\pi} e^{-4\epsilon^2} \sum_{j=j_0}^{N}\int_{I_j} \frac{1}{T-t}\dt= \frac{e^{-2\lambda T}}{2\pi} e^{-4\epsilon^2} (N-j_0) \ln(2).
\end{equation*}
So, given any constant $M>0$, we can always find a finite number $N>0$ such that
\begin{equation*}
I(\omega)\geq M \quad\text{for all $\omega \in A^N$}.
\end{equation*}
Therefore, $\mathbb P(I\geq M)>0$ from \eqref{prodN} and this concludes the proof. 
\end{proof}

\begin{rmk}
Note that the probability of the events $A^N$ appearing in \eqref{prodN} may become very small as $N$ increases. However, since $N$ is finite, this probability is always strictly positive. This is consistent with the fact that $I<+\infty$ a.s.\ and highlights that one cannot obtain any uniform bound for the trajectories in $\Omega$.
\end{rmk}

We conclude this section with the following. 

\begin{proof}[Proof of Lemma \ref{claim}]
For simplicity, for all $i\in\{j_0,\ldots,N\}$, we set
\begin{align}
B_i &:= \left\{ \sup_{t\in I_i}\left|W_t - W_{T-\delta 2^{-i}}\right|
\leq \varepsilon \sqrt{|I_i|} \right\}
\end{align}
therefore
$A_i := \left\{ \left| W_{T-\delta 2^{-i}} - c \right|
\leq \varepsilon \sqrt{|I_i|} \right\}\cap B_i.$

Let $t_i := T - \delta 2^{-i}$ for $i\in\{j_0,\ldots,N\}$, and define the function
\begin{equation}\label{psi}
\psi_i(x)
:=
\mathbb P\!\left(
\bigcap_{k=i}^{N} A_k
\,\middle|\,
W_{t_i}=x
\right),
\qquad i=j_0,\ldots,N.
\end{equation}
The conditioning is understood in the sense of regular conditional
probabilities, which are well defined for almost every $x\in\mathbb R$ with respect to the
law of $W_{t_i}$.
The function $\psi_i$ represents the probability that all remaining events
from $i$ to $N$ are satisfied, given the position of the Brownian motion at
time $t_i$. This interpretation allows us to derive the following
recursive representation
\begin{align}\label{recursion}
\psi_i(x)
&=
\mathbb E\!\left(
\mathbf 1_{\bigl\{ |W_{t_i}-c|\le \varepsilon \sqrt{|I_i|} \bigr\}}
\,\mathbf 1_{B_i}
\,\mathbf 1_{\bigcap_{k=i+1}^N A_k}
\;\middle|\;
W_{t_i}=x
\right)\notag \\
&=
\mathbf 1_{\bigl\{ |x-c|\le \varepsilon \sqrt{|I_i|} \bigr\}}
\,\mathbb E\!\left(
\mathbf 1_{B_i}
\,\mathbf 1_{\bigcap_{k=i+1}^N A_k}
\;\middle|\;
W_{t_i}=x
\right).
\end{align}
The identity \eqref{recursion} follows by noting that,
under the conditioning \(W_{t_i}=x\), the constraint on \(W_{t_i}\) becomes
deterministic. Finally, exploiting the tower property of conditional expectation and the Markov property, we get 
\begin{align*}\label{recursion-1}
\psi_i(x)
&=
\mathbf 1_{\bigl\{ |x-c|\le \varepsilon \sqrt{|I_i|} \bigr\}}
\mathbb E\!\left(
\mathbf 1_{B_i}\,
\psi_{i+1}(W_{t_{i+1}})
\,\middle|\,
W_{t_i}=x
\right),
\qquad i=j_0,\ldots,N-1.
\end{align*}

To obtain the desired result, we look for a strictly positive lower bound for
$\psi_i$. Observe that if
$|x-c|>\varepsilon\sqrt{|I_i|}$, then $\psi_i(x)=0$. For this reason, it suffices to restrict our attention to $x$ close to $c$ and introduce
the interval
\begin{equation}
\tilde I_i := \left[\,c-\frac{\varepsilon \sqrt{|I_i|}}{4},\; c+\frac{\varepsilon \sqrt{|I_i|}}{4}\,\right].
\end{equation}
Observe that if $x\in \tilde I_i$, then $\mathbf 1_{\bigl\{ |x-c|\le \varepsilon \sqrt{|I_i|} \bigr\}}
=1$. We therefore define
\begin{equation}
q_i := \inf_{x\in \tilde I_i} \psi_i(x).
\end{equation}
We now proceed by backward induction to show that the quantities $q_i$ are
strictly positive for all $i\in\{j_0,\ldots,N\}$. For the terminal index $i=N$ and for all $x\in\tilde I_N$, we have $|x-c|\le \varepsilon\sqrt{|I_N|}$. Hence, we note that
\[
\psi_N(x)
= \mathbb P(B_N \mid W_{t_N}=x).
\]
Using the scaling property of Brownian motion we obtain
\begin{align}
\psi_N(x)
&=
\mathbb P\!\left(
\sup_{t\in I_N}\left|W_t-W_{t_N}\right|
\le \varepsilon \sqrt{|I_N|}
\right) \notag\\
&=
\mathbb P\!\left(
\sup_{s\in[0,1]}|\widetilde W_s|\le \varepsilon
\right)
=: \kappa_\varepsilon, \label{kappa}
\end{align}
where
\[
\left\{\frac{W_{t_N+s|I_N|}-W_{t_N}}{\sqrt{|I_N|}}\right\}_{s\in[0,1]}
\overset{d}{\sim}
\{\widetilde W_s\}_{s\in[0,1]},
\]
and $\widetilde W$ is a standard Brownian motion. Here $\kappa_\varepsilon>0$ is a positive constant independent of $N$. This constant is simply the probability that a standard Brownian motion
remains inside $[-\varepsilon,\varepsilon]$ throughout the
time interval $[0,1]$.
This is a classical fact concerning the distribution of the maximum of
Brownian motion, see, for instance, \cite[Section~6.7]{resnick2013adventures}. In particular, $q_N>0$.

Assume now that $q_{i+1}>0$ and let $x\in\tilde I_i$.
We show that $q_i>0$. In more detail,
\begin{align}
\psi_i(x)
&=
\mathbb E\!\left(
\mathbf 1_{B_i}\,
\psi_{i+1}(W_{t_{i+1}})
\,\middle|\,
W_{t_i}=x
\right) \notag\\
&\ge
q_{i+1}\,
\mathbb P\!\left(
B_i\cap\Big\{|W_{t_{i+1}}-c|\leq \tfrac{\varepsilon}{4\sqrt{2}}\sqrt{|I_i|}\Big\}
\,\middle|\,
W_{t_i}=x
\right).
\label{ineq_induction}
\end{align}
Here we used the definition of $q_{i+1}$ together with the fact that
$|I_{i+1}|=\tfrac{1}{2}|I_i|$. Indeed,
\[
|W_{t_{i+1}}-c|\le \tfrac{\varepsilon}{4\sqrt{2}}\sqrt{|I_i|}
= \tfrac{\varepsilon}{4}\sqrt{|I_{i+1}|}
\]
implies that $W_{t_{i+1}}\in \tilde I_{i+1}$, and therefore
\[
\psi_{i+1}(W_{t_{i+1}})\ge q_{i+1}.
\]

To bound the probability appearing in \eqref{ineq_induction}, we apply a
rescaling argument analogous to that used in \eqref{kappa}.
Conditioning on $W_{t_i}=x$ and using the scaling property
of Brownian motion, we obtain
\begin{align}
&\mathbb P\!\left(
B_i\cap\Big\{|W_{t_{i+1}}-c|\leq \tfrac{\varepsilon}{4\sqrt{2}}\sqrt{|I_i|}\Big\}
\,\middle|\,
W_{t_i}=x
\right) \notag\\
&=
\mathbb P\!\left(
\left\{\sup_{s\in[0,1]}|\widetilde W_s|\leq \varepsilon\right\}
\cap
\left\{|\widetilde W_1-m|\leq \tfrac{\varepsilon}{4\sqrt{2}}\right\}
\right),
\end{align}
where $m=\frac{c-x}{\sqrt{|I_i|}}$.
Since $x\in \tilde I_i$, we have
\[
|m|=\frac{|c-x|}{\sqrt{|I_i|}}\le \frac{\varepsilon}{4}.
\]
Consequently,
\[
\Big[-\tfrac{\varepsilon}{4\sqrt{2}}+m,\;
\tfrac{\varepsilon}{4\sqrt{2}}+m\Big]
\subset(-\varepsilon,\varepsilon).
\]

Therefore,
\begin{align}
&\mathbb P\!\left(
\left\{\sup_{s\in[0,1]}|\widetilde W_s|\le \varepsilon\right\}
\cap
\left\{|\widetilde W_1-m|\le \tfrac{\varepsilon}{4\sqrt{2}}\right\}
\right) \notag\\
&=
\mathbb P\big(\sup_{s\in[0,1]}|\widetilde W_s|\le \varepsilon\big)\,
\mathbb P\!\left(
\widetilde W_1\in\Big[-\tfrac{\varepsilon}{4\sqrt{2}}+m,\,
\tfrac{\varepsilon}{4\sqrt{2}}+m\Big]
\,\middle|\,
\sup_{s\in[0,1]}|\widetilde W_s|\le \varepsilon
\right).
\end{align}

Since the conditional law of $\widetilde W_1$ given
$\{\sup_{s\in[0,1]}|\widetilde W_s|\le \varepsilon\}$ admits a continuous
and strictly positive density on $(-\varepsilon,\varepsilon)$, and since
all admissible intervals above are contained in a fixed compact subset of $(-\varepsilon,\varepsilon)$, there exists a constant $c_\varepsilon>0$, depending only on $\varepsilon$, such that
\[
\mathbb P\!\left(
\widetilde W_1\in\Big[-\tfrac{\varepsilon}{4\sqrt{2}}+m,\,
\tfrac{\varepsilon}{4\sqrt{2}}+m\Big]
\,\middle|\,
\sup_{s\in[0,1]}|\widetilde W_s|\le \varepsilon
\right)
\ge c_\varepsilon,
\]
uniformly in $m$. Hence,
\[
\mathbb P\left(\left\{\sup_{s\in[0,1]}|\widetilde W_s|\le \varepsilon\right\}
\cap
\left\{|\widetilde W_1-m|\le \tfrac{\varepsilon}{4\sqrt{2}}\right\}\right)\ge \kappa_\varepsilon\, c_\varepsilon>0.
\]

By the induction hypothesis $q_{i+1}>0$, so it follows that
\[
\psi_i(x)\ge q_{i+1}\kappa_\varepsilon\, c_\varepsilon>0
\qquad \text{for all } x\in \tilde I_i.
\]
Since $q_i:=\inf_{x\in \tilde I_i} \psi_i(x)$, this proves that $q_i>0$. Therefore, $q_i>0$ for all $i\in\{j_0,\ldots,N\}$, and
\[
q_i\ge q_N(\kappa_\varepsilon c_\varepsilon)^{N-i}.
\]
Finally, by the law of total probability and the definition of the regular
conditional probability $\psi_{j_0}$, we obtain
\begin{align*}
\mathbb P\!\left(\bigcap_{k=j_0}^{N}A_k\right)
&=
\mathbb E\!\left(\psi_{j_0}(W_{t_{j_0}})\right) \\
&\ge
\mathbb E\!\left(
\psi_{j_0}(W_{t_{j_0}})
\mathbf 1_{\{W_{t_{j_0}}\in \tilde I_{j_0}\}}
\right) \\
&\ge
q_N (\kappa_\varepsilon c_\varepsilon)^{N-j_0}
\mathbb P\!\left(
|W_{t_{j_0}}-c|
\le \tfrac{\varepsilon \sqrt{\delta\,2^{-j_0-1}}}{4}
\right) \\
&\ge
q_N (\kappa_\varepsilon c_\varepsilon)^{N-j_0}
\tilde{\kappa}_{c,\varepsilon,T,\delta}\,\sqrt{|I_{j_0}|}
>0,
\end{align*}
where $\tilde\kappa_{c,\varepsilon,T,\delta}>0$ is a constant depending only on the parameters $c$, $\varepsilon$, $T$, and $\delta$.
This concludes the proof.
\end{proof}

\appendix

\section{Auxiliary results}
%
%

\begin{lemma}[Backward Grönwall inequality]\label{lem:back_gron}
 Let $g(t)$, $\alpha(t)$, $\beta(t)$ and $\gamma(t)$ be integrable functions with $\beta(t),\gamma(t)\geq 0$. For any $t\in[0,T]$, if 
\begin{equation*}
g(t)\leq \alpha(t)+\beta(t)\int_{t}^{T}\gamma(s)g(s)\d{s}, 
\end{equation*}
then 
\begin{equation*}
g(t)\leq \alpha(t)+\beta(t)\int_{t}^{T}\alpha(u)\gamma(u)e^{\int_{t}^{u}\beta(s)\gamma(s)\d{s}}\d{u}.
\end{equation*}
In particular, if $\alpha(t)\equiv \alpha$, $\beta(t)\equiv \beta$ and $\gamma(t)\equiv 1$, then
\begin{equation*}
g(t)\leq \alpha e^{\beta(T-t)}.
\end{equation*}
\end{lemma}

\begin{lemma}\label{lem:limit}
Let $(X, \mathcal M, \mu)$ be a finite measure space and $f : X \to \mathbb R$ be a measurable function. We assume that there exists $K>0$ such that
\begin{equation}
\label{eq:uniformboundLp}
\forall p \geq 1,\ \norme{f}_{p} \leq K.
\end{equation}
Then $f \in L^{\infty}(X)$ and 
\begin{equation*}
\norme{f}_{\infty} \leq K.
\end{equation*}
\end{lemma}
\begin{proof}
First, we prove that $f \in L^{\infty}(X)$ by proceeding as follows. We argue by contradiction. We assume that $\norme{f}_{\infty} = +\infty$. Let $M>0$, and let us define $A_M = \{|f| \geq M\}$ then we have that $\mu(A_M) >0$, then take $p$ large enough such that $\mu(A_M)^{1/p} \geq 1/2$, then $\norme{f}_{p} \geq (\mu(A_M) M^p)^{1/p} \geq M/2$ and since $M$ is arbitrary, this is contradiction. Thus, this proves that $f \in L^{\infty}(X)$.

Then, we prove the bound. This is a well-known fact that because $f \in L^{\infty}(X)$, 
\begin{equation*}
\lim_{p \to +\infty} \norme{f}_{p} = \norme{f}_{\infty}.
\end{equation*}
Therefore, by using the uniform bound on $L^p$ spaces of $f$ i.e. \eqref{eq:uniformboundLp}, we deduce the result.
\end{proof}

\begin{lemma}\label{lem:nonlinearity}
Let $f\in C^\infty(\mathbb R)$ and define $G(s)=\int_0^1 (1-\sigma) f^{\prime\prime}(\sigma s)\d{\sigma}$.  For any bounded closed interval $\mathcal I\subset \mathbb R$ containing the origin, there are constants $M,M_1>0$ only depending on $f$ and $\mathcal I$ such that
\begin{enumerate}[label=\roman*)]
\item \label{eq:est_sec}  $|G(s)|\leq M$ for all $s\in \mathcal I$.
\item \label{eq:est_dif_G} For any $s_1,s_2\in \mathcal I$, $|s_1^2G(s_1)-s^2_2 G(s_2)|\leq M|s_1-s_2|\left(|s_1|+|s_2|\right)+M_1 \left|s_2\right|^2 |s_1-s_2|$.
\end{enumerate}
\end{lemma}

\begin{proof}
Point \ref{eq:est_sec} is direct with $M:=\max_{\tau\in \mathcal I}|f^{\prime\prime}(\tau)|$ since $f\in C^\infty(\mathbb R)$ and $0\in \mathcal I$. For \ref{eq:est_dif_G}, let $s_1,s_2\in \mathcal I$, then we can write
\begin{equation}\label{eq:rew}
s_1^2G(s_1)-s^2_2 G(s_2)=\left(s_1^2-s_2^2\right)G(s_1)+s_2^2\left(G(s_1)-G(s_2)\right)=:J_1+J_2.
\end{equation}
Using point \ref{eq:est_sec}, we can estimate $J_1$ as follows
\begin{equation}\label{eq:est_J1}
|J_1|=\left|\left(s_1^2-s_2^2\right)\right||G(s_1)|=\left|(s_1-s_2)(s_1+s_2)\right||G(s_1)| \leq M|s_1-s_2|\left(|s_1|+|s_2|\right).
\end{equation}
On the other hand, w.l.o.g assume that $s_1>s_2$. For any $\sigma\in[0,1]$, note that since $f\in C^\infty(\mathbb R)$, we have from mean value theorem that
 \begin{equation}\label{eq:est_int_J2}
|f^{\prime\prime}(\sigma s_1)-f^{\prime\prime}(\sigma s_2)|=\left|\int_{\sigma s_2}^{\sigma s_1}f^\prime(u)\d{u}\right|\leq \sigma \max_{\tau\in \mathcal I}|f^\prime(\tau)| |s_1-s_2| \leq M_1 |s_1-s_2|,
\end{equation}
where we have used that $[\sigma s_2,\sigma s_1]\subset \mathcal I$ since $0\in \mathcal I$ and where $M_1:=\max_{\tau\in \mathcal I}|f^\prime(\tau)|$. Thus, from \eqref{eq:est_int_J2}, we can estimate $J_2$ as 
\begin{align}\notag
|J_2|&=\left|s_2|^2\right|G(s-1)-G(s_2)|=\left|s_2\right|^2\left|\int_{0}^1(1-\sigma)\left[f^{\prime\prime}(\sigma s_1)-f^{\prime\prime}(\sigma s_2)\right]\d{\sigma}\right| \\ \label{eq:est_J2}
&\leq \left|s_2\right|^2 \int_{0}^{1}|1-\sigma| |f^{\prime\prime}(\sigma s_1)-f^{\prime\prime}(\sigma s_2)| \d{\sigma} \leq M_1 \left|s_2\right|^2 |s_1-s_2|.
\end{align}
Putting together \eqref{eq:rew}, \eqref{eq:est_J1} and \eqref{eq:est_J2} yields the desired result. 
\end{proof}

\section*{Acknowledgements}

The authors are grateful to the anonymous referees for their careful reading of the manuscript and for their valuable comments and suggestions, which helped improve the quality and clarity of the paper.

The authors would like to thank to LaSol (Solomon Lefschetz International Research Laboratory) for the financial support. Part of this work was done when the second author was visiting Instituto de Matemáticas, UNAM which provided excellent working conditions. The second author would like to thank Ying Hu for interesting discussions about this work.

\bibliographystyle{alpha}
\small{\bibliography{bibsnse}}

\bigskip

\begin{flushleft}
\textbf{Víctor Hernández-Santamaría}\\
Instituto de Matemáticas\\
Universidad Nacional Autónoma de México \\
Circuito Exterior, C.U.\\
04510, Coyoacán, CDMX, Mexico\\
\texttt{victor.santamaria@im.unam.mx}

\bigskip

\textbf{Kévin Le Balc'h}\\
Laboratoire Jacques-Louis Lions \\
Inria, Sorbonne Université\\
Université de Paris, CNRS \\
Paris, France\\
\texttt{kevin.le-balc-h@inria.fr}

\bigskip

\textbf{Liliana Peralta}\\
Departamento de Matem\'aticas, Facultad de Ciencias\\
Universidad Nacional Autónoma de México \\
Circuito Exterior, C.U.\\
04510, Coyoacán, CDMX, Mexico\\
\texttt{lylyaanaa@ciencias.unam.mx}

\end{flushleft}

\end{document}